\documentclass{amsart}
\usepackage[
	pdfauthor={Bin Guan},
    pdftitle={An Explicit Watson--Ichino Formula with CM Newforms},
				pdfencoding=auto
]{hyperref}

\usepackage{amssymb,amsmath,amscd,amsthm,amsfonts}
\usepackage{graphicx,bbding,pifont,wasysym,skull}
\usepackage{tikz}

\usepackage{chemarrow} 
\usepackage{mathrsfs} 
\usepackage{enumerate}
\usepackage{multicol}

\newcommand{\fp}{\mathfrak{p}}
\newcommand{\CC}{\mathbb{C}}
\newcommand{\A}{\mathbb{A}}
\newcommand{\Q}{\mathbb{Q}}
\newcommand{\R}{\mathbb{R}}
\newcommand{\Z}{\mathbb{Z}}

\newcommand{\cB}{\mathcal{B}}
\newcommand{\cH}{\mathcal{H}}
\newcommand{\cO}{\mathcal{O}}

\newcommand{\ord}{\operatorname{ord}}

\newcommand{\Ad}{\operatorname{Ad}}
\newcommand{\ad}{\operatorname{ad}}
\newcommand{\fin}{\operatorname{fin}}
\newcommand{\1}{\mathbf 1}

\DeclareMathOperator{\Sym}{Sym}
\DeclareMathOperator{\vol}{vol}

\DeclareMathOperator{\St}{St}

\newcommand{\GL}{\operatorname{{GL}}}

\newcommand{\lRS}{\ell_{\mathrm{RS}}}

\newtheorem{Theorem}{Theorem}[section]
\newtheorem{Lemma}[Theorem]{Lemma}

\newtheorem{Proposition}[Theorem]{Proposition}

\newtheorem{Fact}[Theorem]{Fact}

\newtheorem{remark}[Theorem]{Remark}

\begin{document}

\title
{
An Explicit Watson--Ichino Formula with CM Newforms
}
\author{Bin Guan}
\address{Data Science Institute\\Shandong University \\Jinan\\ China}
\email{bguan.math@sdu.edu.cn}

\date{\today}

\begin{abstract}
In this paper, we extend the work 
of Humphries--Khan \cite{humphries2020random} 
to establish an explicit version of Watson--Ichino formula for 
$L(1/2,f\otimes\ad g)$, 
where $f$ is a Hecke--Maass form and $g$ is a CM newform. 
\end{abstract}


\maketitle

\tableofcontents

% ********************************************************************
% ********************************************************************
% ********************************************************************

\section{Introduction}

The purpose of this paper is to establish an explicit Watson--Ichino formula 
for triple product $L$-functions of a special class of Hecke--Maass forms.
To begin with, let $\pi_1,\pi_2,\pi_3$ be 
irreducible unitary cuspidal automorphic representations 
of $\GL_2(\A)$ with the product of their central characters trivial,
where $\A$ is the adele ring over $\Q$.
One can consider the (complete)
triple product $L$-function $L(s,\pi_1\otimes\pi_2\otimes\pi_3)$
associated to them,
which was originally defined classically 
by Garrett \cite{garrett1987decomposition}
and was generalized further 
by Piatetski--Shapiro and Rallis \cite{piatetski1987rankin} 
using an adelic approach. 

Gross and Kudla \cite{gross1992heights}
established an explicit identity relating
central $L$-values and period integrals
(which are finite sums in their case),
when the $\pi_i$'s correspond to cusp forms 
of a prime level and weight $2$.
Watson \cite{watson2008rankin} generalized this identity
to higher levels and weights,
and Ichino \cite{ichino2008trilinear}
proved an adelic version of this period formula which works for
all the cases.
In this paper we study the case when the quaternion algebra 
in \cite{ichino2008trilinear} is $\GL_2$
and the \'etale cubic algebra is $\Q\times\Q\times\Q$ over $\Q$. 
Then the Ichino's formula
can be reformulated (cf. \cite{collins2020anticyclotomic}) as follows:
for any $\varphi_i=\otimes\varphi_{i,v}\in \pi_i$, $i=1,2,3$,
\[\frac
{\left|\int_{\A^\times\GL_2(\Q)\backslash\GL_2(\A)}
\varphi_1(g)\varphi_2(g)\varphi_3(g)\ dg\right|^2}
{\prod_{i=1}^3\int_{\A^\times\GL_2(\Q)\backslash\GL_2(\A)}
|\varphi_i(g)|^2\ dg}
=\frac{C}{2^3}\left(\frac{\pi}6\right)^2
\frac{L(\frac 12,\pi_1\otimes\pi_2\otimes\pi_3)}
{L(1,\pi_1,\Ad)L(1,\pi_2,\Ad)L(1,\pi_3,\Ad)}
\prod_v I_v',\]
where $\A^\times$ is diagonally embedded in $\GL_2(\A)$ as its center,
$C$ is defined so that $dg=C\prod_v d g_v$
is the Tamagawa measure on $\A^\times\GL_2(\Q)\backslash\GL_2(\A)$,
$L(s,\pi_i,\Ad)$ is the adjoint $L$-function attached to $\pi_i$,
and the local constants $I_v'$ are defined as in \eqref{def.localconst}.
Moreover, for $v =p< \infty$,
$I_p'=\vol(\Z_p^\times\backslash\GL_2(\Z_p))$
when all $\pi_{i,p}$ are unramified,
and $\varphi_{i,p}\in \pi_{i,p}$ are unit spherical vectors.

Many authors have derived several explicit versions of Watson--Ichino formula
in various cases. 
For example Nelson \cite{nelson2011equidistribution} extends Watson's formula 
and relates 
\[\left|\int_{\Gamma_0(q)\backslash\mathcal{H}}
y^k f(z)|g(z)|^2\ d\mu(z)\right|^2\quad
\text{ and }
\quad L(\tfrac 12,f\otimes g\otimes\bar g),\]
where $\mathcal{H}$ is the upper-half plane, 
$f$ is a Hecke--Maass form of level $1$ 
and $g$ is a newform of square-free level $q$
($g$ can be holomorphic of weight $k$ or a Maass form).
In this case the triple product $L$-function
$L(s,f\otimes g\otimes\bar g)$ can be factorized as
$L(s,f)L(s,f\otimes\ad g)$
by comparing Euler products.
Recently Humphries and Khan \cite{humphries2020random}
show an exact formula for
$L(1/2,f\otimes\ad g)$ (see \eqref{HK20mainthm}),
where $f$ is a Hecke--Maass form and $g$ is a dihedral Maass newform
(which associates to a Gr\"ossencharacter 
on $\Q(\sqrt{D})$ where 
$D\equiv 1\pmod 4$ is a positive squarefree fundamental discriminant).
With this formula 
\cite{humphries2020random} unconditionally gives a proof 
of the Gaussian moments conjecture
for the fourth moment of dihedral Maass newforms.
This explicit Watson--Ichino formula 
also has some applications in studying quantum variance.
Huang and Lester \cite{huang2020quantum} give 
an asymptotic formula for the harmonic weighted quantum variance of
the family of dihedral Maass forms 
on $\Gamma_0(D)$
with $D$ restricted by some congruence condition.

In this paper, we continue the work 
of Humphries--Khan \cite{humphries2020random} 
and Hu \cite{hu2017triple}
to establish Theorem \ref{WatsonIchinoCM}, 
an explicit version of Ichino's formula for 
$L(1/2,f\otimes\ad g)$, 
where $g=g_\Omega$ is a CM newform  
and $f$ is a Hecke--Maass form. 
These explicit formulas will have, for example, an application 
to quantum variance for CM newforms following the idea
of \cite{huang2020quantum}.

\subsection{Main result}\label{section.mainCM}

Let $q$ be a positive integer 
and $L^2(\Gamma_0(q)\backslash \cH)$
be the space of square-integrable functions on the upper-half plane
$f:\cH\to\CC$ such that $f(\gamma z)=f(z)$ for all $\gamma\in\Gamma_0(q)$.
The inner product in this space is defined by
\[\langle f,g\rangle_q:=\int_{\Gamma_0(q)\backslash \cH}
f(z)\overline{g(z)}\ d\mu(z),
\quad\text{where }d\mu(z)=y^{-2} dx\ dy, \ z=x+iy.\]
But for functions such that $f(\gamma z)=(cz+d)^kf(z)$
with $\gamma=\left(\begin{smallmatrix}a & b \\ c & d\end{smallmatrix}\right)$
(for example, holomorphic modular forms of weight $k$), $\langle f,g\rangle_q$
is defined to be the Petersson inner product whenever the integral converges:
\[\langle f,g\rangle_q:=\int_{\Gamma_0(q)\backslash \cH}
y^kf(z)\overline{g(z)}\ d\mu(z).\]
Notice that
$y^kf(z)\overline{g(z)}$ is $\Gamma_0(q)$-invariant in the latter case,
and the convergency holds if the inner product is defined on cusp forms.

An integer $D$ is a fundamental discriminant 
if either $D\neq 1$, $D\equiv 1\pmod 4$ and $D$ is squarefree,
or $4\mid D$, $\frac D4\equiv 2,3\pmod 4$ and $\frac D4$ is squarefree.
For each $D$ there exists a quadratic extension $E=\Q(\sqrt{D})$ of $\Q$ 
such that $E$ has discriminant $D$.
One can also define a (quadratic)
Dirichlet character modulo $|D|$ by the Kronecker symbol 
$\chi_D(n):=\left(\frac{D}{n}\right)_K$.

For an imaginary quadratic extension $E/\Q$ with 
negative fundamental discriminant $D<0$, 
consider a Hecke character $\Omega$
on $E^\times\backslash\A_E^\times$ 
(its associated classical Gr\"ossencharacter is also denoted by $\Omega$),
whose restriction on $\A^\times$ is trivial.
Assume that it is unramified everywhere at finite places.
At infinity we have $\Omega_\infty=(z/\bar z)^{n}$ for some $n\in\Z$.
Recall that, when $n>0$ and $\Omega$ does not factor 
through the norm map $N_E:\A_E^\times\to\A^\times$,
there is a cuspidal newform 
(cf. \cite{ribet1977galois} and \cite[Theorem 12.5]{iwaniec1997topics})
\[g_\Omega(z):=
\sum_{\mathfrak{a}}\Omega(\mathfrak{a})
(N\mathfrak{a})^n e(zN\mathfrak{a})
\in S_{2n+1}^*(\Gamma_0(|D|),\chi_D)\]
with complex multiplication
of level $|D|$, weight $2n+1$
and nebentypus $\chi_{D}$ such that
$L(s,g_\Omega)=L(s,\Omega)$.
Here the sum is over all integral ideals of $E$,
and $N\mathfrak{a}$ is the norm of $\mathfrak{a}$.

All the $L$-functions in this paper are complete without conductor,
for example,
\[L(s,\pi):=L_\infty(s,\pi)L_{\fin}(s,\pi)\] 
is defined as in \cite{ichino2008trilinear}
(instead of the one in \cite{humphries2020random} that
$\Lambda(s,\pi):=q(\pi)^{s/2}L(s,\pi)$).

We will prove the following
explicit Watson--Ichino formula.
See Theorem \ref{WatsonIchino} for the more general statement.

\begin{Theorem}\label{WatsonIchinoCM}
Let $q_1\mid q=|D|$ with $D<0$ a fundamental discriminant,
$\cB_0^*(q_1)$ be the set of (normalized) Hecke--Maass newforms
of weight $0$ and level $q_1$ with trivial nebentypus.
For any CM newform $g=g_\Omega\in S_{k}^*(\Gamma_0(q),\chi_D)$ 
and Hecke--Maass newform $f\in\cB_0^*(q_1)$ 
normalized such that the Petersson norms are
$\langle g,g\rangle_q=\langle f,f\rangle_q=1$, we have that, 
if $4\nmid q_1$ or $2$ is a Type-2 supercuspidal prime for $f$,
then
\[\begin{split}
\left|\langle f\cdot g,g\rangle_q\right|^2
=\left|\langle f,y^k|g|^2\rangle_q\right|^2
&=\frac{L(\tfrac 12,f)L(\tfrac 12,f\otimes\ad g)}
{L(1,\ad g)^2L(1,\Sym^2 f)}\frac {\nu_{q_1}}{8qq_1\nu_q}\\
&=\frac{L(\tfrac 12,f)L(\tfrac 12,f\otimes\ad g)}
{L(1,\ad g)^2L(1,\Sym^2 f)}
\frac {1}{8q^2}
\prod_{p\mid q,\ p\nmid q_1}(1+p^{-1})^{-1}
\end{split}\]
(notice that 
$|y^{k/2}g(z)|$ is $\Gamma_0(q)$-invariant),
where \[\nu_n:=[\Gamma_0(1):\Gamma_0(n)]=n\prod_{p\mid n}(1+p^{-1});\]
otherwise, when $4\mid q_1$ and $2$ is a Type-1 supercuspidal prime for $f$,
$\left|\langle f,y^k|g|^2\rangle_q\right|^2$
is $L_2(1,\Sym^2 f)=\frac 23$ times above.
For oldforms $(\iota_wf)(z):=f(wz)$ we have the same result for
\[
\langle \iota_{w_1}f,y^k|g|^2\rangle_q
\overline{\langle \iota_{w_2}f,y^k|g|^2\rangle_q}
% =\frac{L(\tfrac 12,f)L(\tfrac 12,f\otimes\ad g)}
% {L(1,\ad g)^2L(1,\Sym^2 f)}
% \frac {1}{8q^2}
% \prod_{p\mid q\text{ but }p\nmid q_1}(1+\tfrac 1p)^{-1},
\]
for any $w_1,w_2\mid \tfrac q{q_1}$.
\end{Theorem}

\begin{remark}
Let $D\equiv 1\pmod 4$ be 
a positive squarefree fundamental discriminant,
and $q_1\mid q=D$. 
\cite[Corollary 4.19]{humphries2020random} shows that,
for any dihedral Maass newform 
$g=g_\Omega\in \cB_0^*(\Gamma_0(q),\chi_D)$ 
and Hecke--Maass newform $f\in\cB_0^*(q_1)$, 
and for any $w_1,w_2\mid \tfrac q{q_1}$,
\begin{equation}\label{HK20mainthm}
\langle \iota_{w_1}f,|g|^2\rangle_q
\overline{\langle \iota_{w_2}f,|g|^2\rangle_q}
=\frac{L(\tfrac 12,f)L(\tfrac 12,f\otimes\ad g)}
{L(1,\ad g)^2L(1,\Sym^2 f)}
\frac {1+\epsilon_f}{2}\frac{\nu_{q_1}}{8qq_1\nu_q}.
\end{equation}
Theorem \ref{WatsonIchino} also leads to the fact that,
the above identity holds when $D$ is 
any positive fundamental discriminant
(either $D$ or $D/4$ is squarefree) and 
$q_1\mid q=D$, $4\nmid q_1$ or 2 is 
not an ``unramified'' dihedral supercuspidal prime (i.e. not a Type-1 prime) for $f$,
except that $\nu_q/\nu_{q_1}=\nu_{q/q_1}$ does not hold in this case.
(If $4\mid q_1$ and $2$ is ``unramified supercuspidal'' for $f$
then the result is multiplied by $\frac 23$.) 

One may notice the extra condition does not show up in \cite{humphries2020random}.
The reason is that they assumed $D$ (and therefore $q_1$) squarefree.
Otherwise, if $4\mid q_1$, that is to say, if $\pi_{f,2}$ is supercuspidal 
(this is the only possible case because $16\nmid D$),
the local $L$-factor $L_p(1,\Sym^2 f)$
at $p=2$ depends on the ``type'' of $\pi_{f,2}$,
which equals either $1$ or $\tfrac 23$.
More details can be found in Section \ref{L(Ad)}.
\end{remark}

\begin{remark}
This explicit Watson--Ichino formula may have some potential applications,
for example, in studying quantum variances, following the idea of
Huang and Lester \cite{huang2020quantum}.
To study the distribution of $L^2$-mass for certain forms,
for example, dihedral Maass forms or CM forms $g_\Omega$,
define
\[\mu_{\Omega}(\psi):=\langle \psi g_\Omega,g_\Omega\rangle_q
=\int_{\Gamma_0(q)\backslash \cH}
y^{k}\psi(z)|g_\Omega(z)|^2\ d\mu(z)\]
for any smooth test function $\psi:\Gamma_0(q)\backslash \cH\to\CC$ with mean zero
which decays rapidly in the cusp,
where $q=|D|$ is the level of $g_\Omega$,
and $k$ the weight of $g_\Omega$.
The proposed quantum variance corresponding to these CM forms 
could be defined by a sum of the form 
\[Q(\psi;K):=\sum_{k\leq K}|\mu_{\Omega}(\psi)|^2\]
as $K\to\infty$.
Theorem \ref{WatsonIchinoCM} gives an explicit formula to write 
the summands as the central values of certain $L$-functions,
when $\psi$ is a Hecke--Maass cuspidal form (old or new).
It is possible to establish an asymptotic formula,
which relates the (harmonic weighted) quantum variance of
the family of CM forms on $\Gamma_0(|D|)$,
to its ``classical variance'' $V(\psi)$.
See \cite{luo2004quantum,huang2020quantum} for more details.
\end{remark}

\subsection{Organization of the paper}

This paper is organized as follows. 
In Section \ref{section.globalcalculation}
we fix notations and normalization,
recall the Watson--Ichino formula
in classical language and 
show how the local constants can be assembled 
into the global result.

By definition a CM form $g$
is associated with a Gr\"ossencharacter $\Omega$
of an imaginary quadratic extension $E/\Q$ with discriminant $D<0$.
When $D$ is squarefree, the main result is nothing new comparing with 
Humphries--Khan's version,
except that the archimedean local constants are
different (see Proposition \ref{localconst.infty}), 
which has been calculated in \cite{watson2008rankin}.
But when $4\mid D$, the general case cannot be avoided
when the levels of $f$ and $g$ are not squarefree.
More precisely, 
for the new case when $\pi_{f,v}$ is supercuspidal,
inspired by the work of Hu \cite{hu2016cuspidal,hu2017triple},
in Section \ref{section.supercuspidal}
we will deal with the Kirillov model of $\pi_{f,v}$
and calculate some special values 
of Whittaker function of a new vector.
Finally,
following ideas 
in \cite{michel2010subconvexity,hu2016cuspidal,humphries2020random}
we will calculate in Section \ref{section.localcalculation}
the local constants $I_v'$
(Propositions 
\ref{localconst.special}, \ref{localconst.unram}, \ref{localconst.supercuspidal})
for $\pi_{f,v}$ special, spherical, and supercuspidal respectively, 
which completes the proof of the Main Theorem \ref{WatsonIchino}.

% ********************************************************************
% ********************************************************************
% ********************************************************************

\section{An explicit version of Watson--Ichino Formula}
\label{section.globalcalculation}

Let $q$ be a positive integer, $q_1\mid q$ and 
$f\in\cB_0^*(q_1)$ be a Hecke--Maass cuspidal newform or 
$f\in S_{k}^*(\Gamma_0(q_1))$ a holomorphic cuspidal newform with trivial nebentypus. 
Define the adelic lift of $f$ by
\[\varphi_f(g):=\big((y^{k/2}f)\mid_k g_\infty\big)(i)\]
with $g=\gamma g_\infty k_0$ given by the strong approximation
where $\gamma\in\GL_2(\Q)$, 
$g_\infty=(\begin{smallmatrix}a&b\\c&d\end{smallmatrix})\in\GL_2^+(\R)$, 
and $k_0\in K_0(q_1)$.
(When $f$ is a newform with nebentypus $\chi_f$, the adelic lift 
is defined to be the above $\varphi_f(g)$ times $\tilde\chi_f(k_0)$,
where $\tilde\chi_f$ is the character of $K_0(q_1)$ given by applying $\chi_f$ to the lower-right entry.)
Let $\pi_f=\otimes_v\pi_{f,v}$ be the cuspidal automorphic representation of $\GL_2(\Q)$
generated by $\varphi_f$.

For finite places $v=p$, we know $\pi_{f,p}$ is 
an unramified principal series representation
if $p\nmid q_1$, and a special representation
(an unramified twist of the Steinberg representation) if $p\parallel q_1$.
When $p^2\mid q_1$, we recall a certain classification of such $\pi_{f,p}$ 
(cf. \cite[Section 2.1.5]{nelson2014bounds}):
\begin{itemize}
	\item \textbf{Type 1.} $\pi_{f,p}$ is an 
	``unramified'' supercuspidal representation,
	i.e. $\pi_{f,p}\simeq\pi_{f,p}\otimes\eta_p$, where 
	$\eta_p$ is the unique nontrivial 
	unramified quadratic character of $\Q_p^\times$.
	Equivalently $\pi_{f,p}$ is a dihedral supercuspidal representation
	associated with 
	an unramified quadratic field extension $E_p/\Q_p$ 
	and a character of $E_p^\times$ that is not trivial 
	on the kernel of the norm map $N_{E_p/\Q_p}:E_p^\times \to \Q_p^\times$.
	In this case we call $p$ a \textbf{Type-1 supercuspidal prime} for $f$.
	(Actually $p$ is called an unramified supercuspidal prime 
	in some other papers, for example, \cite{banerjee19supercuspidal}.
	The reason we rename it in this paper is that, 
	to call it ``unramified'' one might 
	confuse it with the spherical representation.)
	\item \textbf{Type 2.} $\pi_{f,p}$ is supercuspidal satisfying 
	$\pi_{f,p}\not\simeq\pi_{f,p}\otimes\eta_p$,
	with $\eta_p$ above.
	Again in this case we call $p$ a \textbf{Type-2 supercuspidal prime} for $f$.
	\item \textbf{Type 3, 4, 5.} $\pi_{f,p}$ is a ramified principal series,
	or a ramified twist of the Steinberg representation.
\end{itemize}
In this paper we focus on Type 1 and Type 2, 
i.e. $\pi_{f,p}$ is supercuspidal whenever $p^2\mid q_1$.
Actually Types 1 and 2 cover all possibilities in Theorem \ref{WatsonIchinoCM}: 
when $q_1\mid q=|D|$ with $D$ a fundamental discriminant, 
the only possible $p$ such that $p^2\mid q_1$ is $p=2$;
but the conductor exponent of $\pi_{f,p}$ 
is $\geq 4$ for the other three types when $p=2$
(because any character of $\mathbb{Q}_2^\times$ cannot have conductor exponent $1$),
while $16$ can never divide a fundamental discriminant $D$.

\begin{Theorem}\label{WatsonIchino}
Let $q$ be a positive integer, $q_1\mid q$, and 
$\chi$ be a primitive Dirichlet character modulo $q$.
Assume that $f\in\cB_0^*(q_1)$ is a Hecke--Maass newform
satisfying that, 
the corresponding local automorphic representation
$\pi_{f,p}$ is supercuspidal whenever $p^2\mid q_1$.
Then, for any newform $g\in S_{k}^*(q,\chi)$ or $g\in\cB_0^*(q,\chi)$
which is a Hecke eigenform, 
normalized such that 
$\langle g,g\rangle_q=\langle f,f\rangle_q=1$, we have that, 
\[\left|\langle f,y^{k}|g|^2\rangle_q\right|^2
=\frac{L(\tfrac 12,f)L(\tfrac 12,f\otimes\ad g)}
{L(1,\ad g)^2L(1,\Sym^2 f)}
\frac {1}{8q^2}\cdot
\mathcal{C}_\infty\prod_{p}\mathcal{C}_p,\]
where 
\[\mathcal{C}_p=\begin{cases}
1&\text{if }p\nmid q,\\
(1+p^{-1})^{-1}&\text{if }p\mid q,\  p\nmid q_1,\\
1&\text{if }p\mid q,\ p\parallel q_1,\\
(1+p^{-1})^{-1}&\text{if }p\mid q,\ p^2\mid q_1,
\text{ $p$ is a Type-1 supercuspidal prime for $f$},\\
1&\text{if }p\mid q,\ p^2\mid q_1,
\text{ $p$ is a Type-2 supercuspidal prime for $f$};
\end{cases}\]
and $\mathcal{C}_\infty=\begin{cases}1&k>0\\\tfrac{1+\epsilon_f}{2}&k=0\end{cases}$
with $\epsilon_f\in\{\pm 1\}$ the parity of $f$.
For oldforms $(\iota_wf)(z):=f(wz)$ we have the same result, i.e. 
\[
\langle \iota_{w_1}f,y^{k}|g|^2\rangle_q
\overline{\langle \iota_{w_2}f,y^{k}|g|^2\rangle_q}
=\frac{L(\tfrac 12,f)L(\tfrac 12,f\otimes\ad g)}
{L(1,\ad g)^2L(1,\Sym^2 f)}
\frac {1}{8q^2}\cdot
\mathcal{C}_\infty\prod_{p}\mathcal{C}_p,
\]
for any $w_1,w_2\mid \tfrac q{q_1}$, any normalized Hecke--Maass newform $f\in\cB_0^*(q_1)$,
and any normalized Hecke newform $g\in S_{k}^*(q,\chi)$ or $g\in\cB_0^*(q,\chi)$.
\end{Theorem}

\begin{proof}
Using the notations 
in \cite[Section 4.2]{humphries2020random},
we denote by $\varphi_1,\varphi_2,\varphi_3$
the adelic lifts of $g,\bar g,\iota_wf$ respectively,
and by $\pi_i$ the cuspidal automorphic representation of $\GL_2$
generated by $\varphi_i$ ($i=1,2,3$).
Here we have $\pi_2=\tilde\pi_1$ (the contragredient).
Let $\tilde g\in S_{k}^*(q,\bar\chi)$ or $\cB_0^*(q,\bar\chi)$ 
be a Hecke eigenform 
such that $g$ and $\overline{\tilde g}$ are
both associated to the same newform,
and $\tilde\varphi_1$ be
the adelic lifts of $\tilde g$ 
(and define $\tilde\varphi_2,\tilde\varphi_3$ respectively).

We fix a $\GL(2,\Q_v)$-invariant bilinear local pairing $\langle\cdot,\cdot\rangle$
on $\pi_{i,v}\otimes\tilde\pi_{i,v}$ for each place $v$ and each $i = 1, 2, 3$, 
and use this to define a pairing $\langle\cdot,\cdot\rangle$ on
$\Pi_{v}\otimes\tilde\Pi_{v}$ (where $\Pi_v=\pi_{1,v}\otimes\pi_{2,v}\otimes\pi_{3,v}$)
determined on simple tensors $\varphi_v:=\varphi_{1,v}\otimes
\varphi_{2,v}\otimes\varphi_{3,v}$ and 
$\tilde\varphi_v:=\tilde\varphi_{1,v}\otimes
\tilde\varphi_{2,v}\otimes\tilde\varphi_{3,v}$
by 
\[\langle\varphi_{v},\tilde\varphi_{v}\rangle
:=\langle\varphi_{1,v},\tilde\varphi_{1,v}\rangle
\langle\varphi_{2,v},\tilde\varphi_{2,v}\rangle
\langle\varphi_{3,v},\tilde\varphi_{3,v}\rangle.\]
Note that this is unique up to nonzero scalar. 
Then we define
\begin{gather}
I_v(\varphi_v\otimes\tilde\varphi_v)
:=\int_{Z(F_v)\backslash\GL_2(F_v)}
\prod_{i=1}^3\langle\pi_{i,v}(g_v)\varphi_{i,v},
\tilde\varphi_{i,v}\rangle\ dg_v,\nonumber\\
I'_v(\varphi_v\otimes\tilde\varphi_v)
:=\frac{L(1,\pi_{1,v},\Ad)L(1,\pi_{2,v},\Ad)L(1,\pi_{3,v},\Ad)}
{\zeta_{F_v}(2)^2L(\tfrac 12,\pi_{1,v}\otimes\pi_{2,v}\otimes\pi_{3,v})}
\frac{I_v(\varphi_v\otimes\tilde\varphi_v)}
{\langle\varphi_{v},\tilde\varphi_{v}\rangle}.\label{def.localconst}
\end{gather}

We follow the normalization of local Haar measures
in \cite[Section 4.2]{humphries2020random}.
That is, the Haar measure $dg_v$ on $Z(F_v)\backslash\GL_2(F_v)$
at any non-archimedean place
is defined such that,
under the decomposition of $dg_v$ induced by the Iwasawa decomposition,
the maximal compact subgroup $\GL_2(\Z_p)$
has volume $1$;
and the Haar measure at any real place is
$dg_v:=dx_v\cdot|y_v|_v^{-1}d^\times y_v\cdot dk_v$
with $g_v=(\begin{smallmatrix}y_v&x_v\\0&1\end{smallmatrix})k_v$, $k_v\in K_v$,
where $dk_v$ is the Haar measure on $K_v=SO(2)$ with volume $1$.

The Watson--Ichino formula gives
\begin{multline}\label{ClassicWatsonIchino}
\int_{\Gamma_0(q)\backslash\mathcal{H}}
y^{k}|g(z)|^2(\iota_{w_1}f)(z)\ d\mu(z)
\overline{\int_{\Gamma_0(q)\backslash\mathcal{H}}
y^{k}|g(z)|^2(\iota_{w_2}f)(z)\ d\mu(z)}\\
=\frac{\mathcal{C}_\infty}{8\nu_q}
\frac{L(\tfrac 12,f)L(\tfrac 12,f\otimes\ad g)}
{L(1,\ad g)^2L(1,\Sym^2 f)}
\prod_{p\mid q}I_p'(\varphi_p\otimes\tilde\varphi_p),
\end{multline}
where $\nu_q=q\prod_{p\mid q}(1+p^{-1})$. 
This formula differs with 
the one given in \cite[Section 4.3]{humphries2020random},
because the local constant at infinity becomes 
$I'_\infty(\varphi_\infty\otimes\tilde\varphi_\infty)
=\mathcal{C}_\infty$,
given by Proposition \ref{localconst.infty},
and also because the $L$-functions in \cite{humphries2020random}
are defined with conductors.

Notice that (cf. \cite{watson2008rankin} or
\cite{loeffler2012computation}) for $p\mid q$,
the local component $\pi_{1,p}$ of $g$ 
is a unitarizable ramified principal series representation 
$\omega_{1,p}\boxplus\omega_{2,p}$, 
where the unitary characters 
$\omega_{1,p},\omega_{2,p}$ of $\Q_p^\times$
have conductor exponents 
\[c(\omega_{1,p})=c((\chi)_p)=\ord_p(q)>0
\quad\text{and}\quad
c(\omega_{2,p})=0;\]
and $\pi_{2,p}=\tilde\pi_{1,p}=
\omega_{2,p}^{-1}\boxplus\omega_{1,p}^{-1}$.
(Here $(\chi)_p$ is the local component of the Hecke character
corresponding to $\chi$.)
Also, $\varphi_{1,p},\varphi_{2,p},
\tilde\varphi_{1,p},\tilde\varphi_{2,p}$ 
are all local newforms in corresponding representations.
However, in this paper, 
the assumption that $q$ might not be squarefree,
leads to more cases for $\varphi_p\otimes\tilde\varphi_p$ than
those listed in the proof of 
\cite[Corollary 4.19]{humphries2020random}.
We list all the cases for $\pi_{3,p}$ as follows. 
\begin{enumerate}[(i)]
	\item When $p\parallel q_1$, 
	the local component $\pi_{3,p}$ of $f$ 
	is a special representation $\St_{\omega_{3,p}}$, 
	where $\omega_{3,p}$ is either the trivial character 
	or the unramified quadratic character of $\Q_p^\times$. 
	\item When $p\nmid q_1$, 
	the local component $\pi_{3,p}$ of $f$ 
	is a unitarizable unramified principal series representation 
	$\omega_{3,p}\boxplus\omega_{3,p}^{-1}$, 
	where $p^{-1/2} < |\omega_{3,p}(p)| < p^{1/2}$ and $c(\omega_{3,p}) = 0$.
	\item When $p^2\mid q_1$, under our assumption of this theorem,
	the local component $\pi_{3,p}$ of $f$ 
	is a supercuspidal representation with 
	trivial central character and $c(\pi_{3,p})=\ord_p(q_1)$.
\end{enumerate}
In all these cases $\varphi_{3,p}$ and $\tilde\varphi_{3,p}$
are translates of local newforms by 
$\pi_{3,p}(\begin{smallmatrix}w_1^{-1}&0\\0&1\end{smallmatrix})$ 
and $\tilde\pi_{3,p}(\begin{smallmatrix}w_2^{-1}&0\\0&1\end{smallmatrix})$
respectively.
(When $p\nmid w_1$, $\varphi_{3,p}$ is just the local newform;
and so is it for $w_2$ and $\tilde\varphi_{3,p}$.)

We respectively apply Propositions 
\ref{localconst.special}, \ref{localconst.unram}, \ref{localconst.supercuspidal}
with $F_v = \Q_p$, $q_v = p$ and $m_v=\ord_p(q)$ 
to give the local constants $I_p'(\varphi_p\otimes\tilde\varphi_p)$.
\end{proof}

The following propositions determine all the local constants $I_v'$
we need in the above proof. 

\begin{Proposition}[{\cite[Theorem 3]{watson2008rankin}}]\label{localconst.infty}
For $F_v\simeq \R$, let $k(\pi_{v})\in\Z$ denote the weight of $\pi_{v}$ 
and let $\epsilon\in \{1, i, -1, -i\}$ denote the local root number. Then
\[I_v'(\varphi_v\otimes\tilde\varphi_v)=\begin{cases}
1&\text{if }k(\pi_{1,v})=-k(\pi_{2,v})>k(\pi_{3,v})=0,\\
\tfrac{1+\epsilon_1\epsilon_2\epsilon_3}{2}&
\text{if }k(\pi_{1,v})=k(\pi_{2,v})=k(\pi_{3,v})=0.
\end{cases}\]
\end{Proposition}

Now let $F_v$ be a nonarchimedean local field 
with uniformizer $\varpi_v$ and cardinality $q_v$ of the residue field. 
The proof of the followings can be found in the next section.

\begin{Proposition}
[cf. {\cite[Proposition 4.16]{humphries2020random}}]\label{localconst.special}
Let $\pi_{1,v}=\omega_{1,v}\boxplus\omega_{2,v}$
and $\pi_{2,v}=\tilde\pi_{1,v}
=\omega_{2,v}^{-1}\boxplus\omega_{1,v}^{-1}$
be principal series representations of $\GL_2(F_v)$ 
for which the characters $\omega_{1,v},\omega_{2,v}$ of $F_v^\times$
have levels (i.e. conductor exponents) $c(\omega_{1,v}) = m_v>0$ 
and $c(\omega_{2,v}) = 0$, 
and let $\pi_{3,v}=\St_{\omega_{3,v}}$ 
be a special representation with $c(\omega_{3,v}) = 0$ 
and $\omega_{3,v}^2 = \1$. 
Suppose that $\pi_{1,v},\pi_{2,v},\pi_{3,v}$ 
are irreducible and unitarizable, 
so that $\omega_{1,v},\omega_{2,v},\omega_{3,v}$ are unitary. 
Then if $\varphi_{1,v},\varphi_{2,v},\varphi_{3,v},
\tilde\varphi_{1,v},\tilde\varphi_{2,v},\tilde\varphi_{3,v}$ 
are all local newforms,
\[I_v'(\varphi_v\otimes\tilde\varphi_v)
=q_v^{-m_v}(1+q_v^{-1}).\]
This also holds if either or both 
$\varphi_{3,v}$ and $\tilde\varphi_{3,v}$ 
are translates of local newforms by 
$\pi_{3,v}(\begin{smallmatrix}\varpi_v^{-l_1}&0\\0&1\end{smallmatrix})$ 
and $\tilde\pi_{3,v}(\begin{smallmatrix}\varpi_v^{-l_2}&0\\0&1\end{smallmatrix})$ 
respectively, where $0\leq l_1,l_2\leq m_v-1$.
\end{Proposition}

\begin{Proposition}
[cf. {\cite[Proposition 4.17]{humphries2020random}}]\label{localconst.unram}
Let $\pi_{1,v}=\omega_{1,v}\boxplus\omega_{2,v}$,
$\pi_{2,v}=\tilde\pi_{1,v}
=\omega_{2,v}^{-1}\boxplus\omega_{1,v}^{-1}$,
and $\pi_{3,v}=\omega_{3,v}\boxplus\omega_{3,v}^{-1}$ 
be principal series representations of $\GL_2(F_v)$ 
with $c(\omega_{1,v}) = m_v>0$ and
$c(\omega_{2,v}) = c(\omega_{3,v}) = 0$. 
Suppose that $\pi_{1,v},\pi_{2,v},\pi_{3,v}$ 
are irreducible and unitarizable, so that 
$\omega_{1,v},\omega_{2,v},\omega_{3,v}$ are unitary while 
$q^{-1/2} < |\omega_{3,v}(\varpi_v)| < q^{1/2}$. 
Then if $\varphi_{1,v},\varphi_{2,v},\varphi_{3,v},
\tilde\varphi_{1,v},\tilde\varphi_{2,v},\tilde\varphi_{3,v}$ 
are all local newforms,
\[I_v'(\varphi_v\otimes\tilde\varphi_v)
=q_v^{-m_v}.\]
This also holds if either or both 
$\varphi_{3,v}$ and $\tilde\varphi_{3,v}$ 
are translates of local newforms by 
$\pi_{3,v}(\begin{smallmatrix}\varpi_v^{-l_1}&0\\0&1\end{smallmatrix})$ 
and $\tilde\pi_{3,v}
(\begin{smallmatrix}\varpi_v^{-l_2}&0\\0&1\end{smallmatrix})$ 
respectively, where $0\leq l_1,l_2\leq m_v$.
\end{Proposition}

\begin{Proposition}\label{localconst.supercuspidal}
Let $\pi_{1,v}=\omega_{1,v}\boxplus\omega_{2,v}$,
$\pi_{2,v}=\tilde\pi_{1,v}
=\omega_{2,v}^{-1}\boxplus\omega_{1,v}^{-1}$ be as above,
and $\pi_{3,v}$ 
be a supercuspidal representation of $\GL_2(F_v)$ 
with $c(\pi_{3,v}) = c_v\leq m_v$. 
Suppose that $\pi_{1,v},\pi_{2,v},\pi_{3,v}$ 
are irreducible and unitarizable, so that 
$\omega_{1,v},\omega_{2,v},\omega_{3,v}$ are unitary. 
Then if $\varphi_{1,v},\varphi_{2,v},\varphi_{3,v},
\tilde\varphi_{1,v},\tilde\varphi_{2,v},\tilde\varphi_{3,v}$ 
are all local newforms,
\[I_v'(\varphi_v\otimes\tilde\varphi_v)
=\begin{cases}
q_v^{-m_v}(1+q_v^{-1})&\text{if }\pi_{3,v}\not\simeq\pi_{3,v}\otimes\eta_v,\\
q_v^{-m_v}&\text{if }\pi_{3,v}\simeq\pi_{3,v}\otimes\eta_v,
\end{cases}\]
where $\eta_v$ is the (nontrivial) unramified quadratic character of $F_v^\times$.
This also holds if either or both 
$\varphi_{3,v}$ and $\tilde\varphi_{3,v}$ 
are translates of local newforms by 
$\pi_{3,v}(\begin{smallmatrix}\varpi_v^{-l_1}&0\\0&1\end{smallmatrix})$ 
and $\tilde\pi_{3,v}
(\begin{smallmatrix}\varpi_v^{-l_2}&0\\0&1\end{smallmatrix})$ 
respectively, where $0\leq l_1,l_2\leq m_v-c_v$.
\end{Proposition}

% ********************************************************************
% ********************************************************************
% ********************************************************************

\section{Local calculation in the Watson--Ichino Formula}
\label{section.localcalculation}

Let $F$ (in this section we drop all the subscripts $v$) 
be a nonarchimedean
local field with ring of integers $\cO_F$, uniformizer $\varpi$, 
and maximal ideal $\fp = \varpi\cO_F$.
Let $q := N(\fp) = |\varpi|^{-1}$, where the norm $| \cdot |$ 
is such that $|x| = q^{-v(x)}$ for $x \in \varpi^{v(x)}\cO_F^\times$.

Let $G:=\GL_2(F)$, $K:= \GL_2(\cO_F)$ and define the congruence subgroup
\[K_1(\fp^m):=\left\{k\in K:
k\equiv\begin{pmatrix}*&*\\0&1\end{pmatrix}
\pmod{\fp^m}\right\}\]
for any nonnegative integer $m$. 
We normalize the additive Haar measure $dx$ on $F$, 
the multiplicative Haar measure 
$d^\times x:= \zeta_F (1)|x|^{-1} dx$ on
$F^\times$, and the Haar measure $dk$ on $K$
so that 
\[\vol(\cO_F;dx)=1,\quad 
\vol(\cO_F^\times;d^\times x)=1,\quad
\vol(K;dk)=1,\]
with $\zeta_F (s): = (1 - q^{-s})^{-1}$.
Denote by $Z$ the center of $G$,
by $A$ the diagonal subgroup with lower diagonal entry equal to $1$,
and by $N$ the usual upper triangular unipotent
subgroup of $G$.
Denote by $B:=ZAN$ the usual Borel subgroup of $G$.
For $t,y\in F^\times$ and $x\in F$,
we set
\[w:=\begin{pmatrix}0&-1\\1&0\end{pmatrix},\quad
z(t):=\begin{pmatrix}t&0\\0&t\end{pmatrix},\quad
a(y):=\begin{pmatrix}y&0\\0&1\end{pmatrix},\quad
n(x):=\begin{pmatrix}1&x\\0&1\end{pmatrix}.\]

% ********************************************************************
% ********************************************************************
% ********************************************************************

\subsection{Whittaker models}

Let $(\pi,V_\pi)$ be an irreducible admissible smooth representation of $G$.
Let $c(\pi)$ be the level (or the conductor exponent) of $\pi$,
which is the smallest nonnegative integer such that
$\pi^{K_1(\fp^{c(\pi)})}\neq 0$.
In this case the invariant space is $1$-dimensional, 
and we call a nontrivial vector in this subspace a newform in $\pi$. 
In this section
$\varphi_{\pi_i}\in\pi_i$ and $\tilde\varphi_{\pi_i}\in\tilde\pi_i$ are newforms 
unless otherwise specified.

Fix a nontrivial continuous additive character $\psi$ of $F$.
Assume that $\psi$ is unramified in this paper,
i.e. the smallest integer $c(\psi)$ such that 
$\psi$ is trivial on $\fp^{c(\psi)}$ is $0$.
Let $\mathcal{W}(\psi)$ be the space of all smooth Whittaker functions,
i.e. all smooth functions $W(g)$ on $G$ satisfying 
\[W(n(x)g)=\psi(x)W(g)\quad\text{ for all }n(x)\in N.\]
If $\pi$ is generic, 
i.e. there is a nontrivial intertwining map $V_\pi\to \mathcal{W}(\psi)$,
we denote the image by $\mathcal{W}(\pi,\psi)$
and call it the Whittaker model of $\pi$.

For generic irreducible unitarizable representations 
$\pi_1,\pi_2,\pi_3$ with $\pi_1$ a principal series representation, 
and for $\varphi_1$ in the induced model of $\pi_1$, 
$W_2\in\mathcal{W}(\pi_2,\bar\psi)$, and
$W_3\in\mathcal{W}(\pi_3,\psi)$, 
we define the local Rankin--Selberg integral by
\[\lRS(\varphi_1, W_2, W_3)
:=\zeta_F(1)^{1/2}\int_K\int_{F^\times}
\varphi_1(a(y)k)W_2(a(y)k)W_3(a(y)k)
\ \frac{d^\times y}{|y|}\ dk.\]
Michel and Venkatesh \cite{michel2010subconvexity}
show a result that relates $\lRS$ and 
the local constants $I(\varphi\otimes\tilde\varphi)$ 
in the Watson--Ichino formula.

\begin{Lemma}[{\cite[Lemma 3.4.2]{michel2010subconvexity}},
{\cite[Lemma 5.2]{humphries2020random}}]
For $g,h\in G$,
$\varphi=\varphi_{\pi_1}\otimes\varphi_{\pi_2}\otimes
\pi_3(g)\varphi_{\pi_3}$ and
$\tilde\varphi=\tilde\varphi_{\pi_1}\otimes
\tilde\varphi_{\pi_2}\otimes
\tilde\pi_3(h)\tilde\varphi_{\pi_3}$ 
with $\varphi_{\pi_1},\varphi_{\pi_2},\varphi_{\pi_3},
\tilde\varphi_{\pi_1},\tilde\varphi_{\pi_2},\tilde\varphi_{\pi_3}$ 
newforms, we have
\[I(\varphi\otimes\tilde\varphi)=
\lRS(\varphi_{\pi_1},W_{\pi_2},\pi_3(g)W_{\pi_3})
\lRS(\tilde\varphi_{\pi_1},\widetilde{W}_{\pi_2},
\tilde\pi_3(h)\widetilde{W}_{\pi_3})\]
whenever $\pi_2$ is tempered.
\end{Lemma}

Notice that both $\varphi_\pi$ and $W_\pi$
are $K_1(\fp^{c(\pi)})$-invariant.
The following lemma, together with Lemma \ref{Lemma.BKinvariance}, 
reduces the calculation of local constants
to determining the values of these functions 
at \[g=\begin{pmatrix}y&0\\0&1\end{pmatrix}
\begin{pmatrix}1&0\\\varpi^j&1\end{pmatrix}
\quad\text{for }0\leq j\leq c(\pi).\]

\begin{Lemma}[{cf. \cite[Lemma 2.2]{hu2016cuspidal}}]\label{Lemma.LocalConstCalculation}
Fix an integer $m\geq 0$.
For any left $(B\cap K)$-invariant and right $K_1(\fp^m)$-invariant 
function $\Theta:K\to\CC$, if integrable, we have
\[\int_K \Theta(k)\ dk=\sum_{j=0}^mA_j
\Theta\left(\begin{pmatrix}1&0\\\varpi^j&1\end{pmatrix}\right),
\quad \text{where }
A_j=\frac{\zeta_F(2)}{\zeta_F(1)}\cdot
\begin{cases}
1,&\text{if }j=0,\\q^{-j}\zeta_F(1)^{-1},&\text{if }0<j<m,\\
q^{-m},&\text{if }j=m.
\end{cases}\]
\end{Lemma}
\begin{proof}
By the same way of proving \cite[Lemma 2.2]{hu2016cuspidal},
one can also show that, for any right $K_1(\fp^m)$-invariant 
function $\Theta:G=\GL_2(F)\to\CC$, if integrable, we have
\begin{equation}\label{eqn.hu2016cuspidal}
\int_G \Theta(g)\ dg=\sum_{j=0}^mA_j
\int_B\Theta\left(b\begin{pmatrix}1&0\\\varpi^j&1\end{pmatrix}\right)\ db,
\end{equation}
with $A_j$ defined as in the above lemma, 
where $dg$ is the normalized Haar measure on $G$ such that $K$ has volume $1$,
and $db$ is the left Haar measure on $B$ such that $B\cap K$ has volume $1$.
Lemma \ref{Lemma.LocalConstCalculation} is a direct corollary of the above formula.
\end{proof}
\begin{remark}
The generalization in \cite[Lemma 5.18]{humphries2020random} of \eqref{eqn.hu2016cuspidal}, 
which says that Lemma \ref{Lemma.LocalConstCalculation} holds for any right $K_1(\fp^m)$-invariant function,
is wrong. In fact, 
$\{(\begin{smallmatrix}1&0\\\varpi^j&1\end{smallmatrix}):0\leq j\leq m\}$
is not a complete coset representatives for $K/K_1(\fp^m)$;
one can show that $[K:K_1(\fp^m)]=\zeta_F(2)^{-1}q^{2m}$ for $m\geq 1$. 
Luckily, the functions they integrate in Section 5.3 of \cite{humphries2020random} 
are actually left $(B\cap K)$-invariant,
so their calculations work well.
See Lemma \ref{Lemma.BKinvariance} for more details.
\end{remark}

We are interested in the following cases:
$\pi_1=\omega_1\boxplus\omega_2$, 
$\pi_2=\omega_2^{-1}\boxplus\omega_1^{-1}$
are principal series representations
with $\omega_1,\omega_2$ both unitary,
$c(\omega_1)=c(\chi_D)$ and $c(\omega_2)=0$,
so that $c(\pi_1)=c(\pi_2)=c(\chi_D)$;
and $\pi_3$ is one of the following cases:
\begin{itemize}
\item a special representation 
$\St_{\omega_3}$ with $\omega_3$ unitary and unramified 
and $\omega_3^2=\1$, or 
\item a principal series representation
$\omega_3\boxplus\omega_3^{-1}$
with $q^{-1/2}\leq|\omega_3(\varpi)|<q^{1/2}$
and $c(\omega_3)=0$ so that $c(\pi_3)=0$, or
\item a supercuspidal representation 
with trivial central character and
$c(\pi_3)\leq c(\chi_D)$.
\end{itemize} 
In particular, the central character of $\Pi=\pi_1\otimes\pi_2\otimes\pi_3$ is trivial,
so $\Pi$ is self dual and $\overline{\Pi}\simeq\widetilde{\Pi}$. 
One can take the newforms $\tilde\varphi_{\pi_i}$ so that
$\tilde\varphi_{\pi_1}\otimes
\tilde\varphi_{\pi_2}\otimes\tilde\varphi_{\pi_3}
=\bar\varphi_{\pi_1}\otimes
\bar\varphi_{\pi_2}\otimes\bar\varphi_{\pi_3}$
in both the induced and Whittaker models.

Next we will calculate the values of Whittaker functions case by case.

% ********************************************************************
% ********************************************************************
% ********************************************************************

\subsection{Whittaker functions for induced representations}

For a principal series representation $\pi=\omega\boxplus\omega'$ 
or a special representation $\pi=\St_\omega$, and given a vector
$\varphi_\pi$ in the induced model of $\pi$,
denote by 
\[W_\pi(g):=\frac{\zeta_F(2)^{1/2}}{\zeta_F(1)}
\int_F\varphi_\pi(w\cdot n(x)\cdot g)\psi^{-1}(x)\ dx\]
the corresponding element in the Whittaker model $\mathcal{W}(\pi,\psi)$.
(This differs with the definition in \cite{humphries2020random} 
by an inverse of $\psi$,
so that $W(n(x)g)=\psi(x)W(g)$ holds.)
Here the normalization of $W_\pi$ follows 
\cite[Section 3.2.1]{michel2010subconvexity} so that
the map $\varphi_\pi\mapsto W_\pi$ is isometric,
where the invariant bilinear pairings on $\pi\otimes\tilde\pi$ 
on the induced model and the Whittaker model
are defined respectively by
\[\langle\varphi_\pi,\tilde\varphi_\pi\rangle
:=\int_K\varphi_\pi(k)\tilde\varphi_\pi(k)\ dk,\quad
\langle W_\pi,\widetilde{W}_\pi\rangle
:=\int_{F^\times}W_\pi(a(y))\widetilde{W}_\pi(a(y))\ d^\times y\] 
with $dk$ the Haar measure on $K$ such that $\vol(K)=1$.

For $\pi_1=\omega_1\boxplus\omega_2$, 
$\pi_2=\omega_2^{-1}\boxplus\omega_1^{-1}$
with $c(\omega_1)=m>0$ and $c(\omega_2)=0$,
we recall the following results.

\begin{Lemma}[\cite{schmidt02remarks}]\label{Lemma.newform1}
The newform in the induced model of $\pi_1$ is given by
\[\varphi_{\pi_1}(g)=\begin{cases}
\omega_1(a)\omega_2(d)\left|\dfrac ad\right|^{1/2}&\text{if }
g\in\begin{pmatrix}a&b\\0&d\end{pmatrix}
\begin{pmatrix}1&0\\1&1\end{pmatrix}
K_1(\fp^m),\\
0&\text{if }
g\in\begin{pmatrix}a&b\\0&d\end{pmatrix}
\begin{pmatrix}1&0\\\varpi^j&1\end{pmatrix}
K_1(\fp^m)\text{ for some }0<j\leq m.
\end{cases}\]
Its corresponding Whittaker function has 
$W_{\pi_3}(\begin{smallmatrix}1&0\\0&1\end{smallmatrix})
=\frac{\zeta_F(2)^{1/2}}{\zeta_F(1)}$; for any $y\in F^\times$,
\[W_{\pi_1}(a(y))=\frac{\zeta_F(2)^{1/2}}{\zeta_F(1)}\cdot
\begin{cases}
\omega_2(y)|y|^{1/2}&\text{if }v(y)\geq 0,\\
0&\text{if }v(y)< 0;\end{cases}\]
by taking complex conjugates 
(so that $W_2\in\mathcal{W}(\pi_2,\bar\psi)$) we have 
\[W_{\pi_2}(a(y))=\frac{\zeta_F(2)^{1/2}}{\zeta_F(1)}\cdot
\begin{cases}
\omega_2^{-1}(y)|y|^{1/2}&\text{if }v(y)\geq 0,\\
0&\text{if }v(y)< 0.\end{cases}\]
\end{Lemma}

Now we work on the values of 
$W_{\pi_1}\left(a(y)(\begin{smallmatrix}1&0\\\varpi^j&1\end{smallmatrix})\right)$
for $0\leq j<m$, here $c(\omega_1)=m>0$, $c(\omega_2)=0$.
% Recall that \cite[Lemma 5.12]{humphries2020random}
% proves the following lemma for $j=0$.

\begin{Lemma}[cf. \cite{hu2017triple,humphries2020random}]
\label{Lemma.newform1.Whittaker}
We have that
\[W_{\pi_1}\left(a(y)\begin{pmatrix}1&0\\1&1\end{pmatrix}\right)
=\frac{\zeta_F(2)^{1/2}}{\zeta_F(1)}\begin{cases}
\omega_1(y)|y|^{1/2}\psi(y)\epsilon(1,\omega_1\omega_2^{-1},\psi^{-1})
&\text{if }v(y)\geq -m,\\
0&\text{if }v(y)<-m.
\end{cases}\]
\[W_{\pi_2}\left(a(y)\begin{pmatrix}1&0\\1&1\end{pmatrix}\right)
=\frac{\zeta_F(2)^{1/2}}{\zeta_F(1)}\begin{cases}
\omega_1^{-1}(y)|y|^{1/2}\psi(-y)\epsilon(1,\omega_1^{-1}\omega_2,\psi)
&\text{if }v(y)\geq -m,\\
0&\text{if }v(y)<-m.
\end{cases}\]
(Recall that our definition of $W_{\pi}$ differs by an inverse 
with that in \cite{humphries2020random}.)
And for $0<j<m$,
\[W_{\pi_1}\left(a(y)\begin{pmatrix}1&0\\\varpi^j&1\end{pmatrix}\right)
=\frac{\zeta_F(2)^{1/2}}{\zeta_F(1)}\begin{cases}
\omega_2(y)|y|^{1/2}\int_{\cO_F}
\omega_1^{-1}\omega_2(1+x\varpi^j)\psi(-xy)\ dx
&\text{if }v(y)=j-m,\\
0&\text{if }v(y)\neq j-m.
\end{cases}\]
\[W_{\pi_2}\left(a(y)\begin{pmatrix}1&0\\\varpi^j&1\end{pmatrix}\right)
=\frac{\zeta_F(2)^{1/2}}{\zeta_F(1)}\begin{cases}
\omega_2^{-1}(y)|y|^{1/2}\int_{\cO_F}
\omega_1\omega_2^{-1}(1+x\varpi^j)\psi(xy)\ dx
&\text{if }v(y)=j-m,\\
0&\text{if }v(y)\neq j-m.
\end{cases}\]
\end{Lemma}
\begin{proof}
Let \[g=w\cdot n(x)\cdot
a(y)\cdot\begin{pmatrix}1&0\\\varpi^j&1\end{pmatrix}
=\begin{pmatrix}-\varpi^j&-1\\y+x\varpi^j&x\end{pmatrix}.\]
When $j=0$,
\begin{equation}\label{matrixdecomp.j=0}
g=\begin{pmatrix}-1&-1\\y+x&x\end{pmatrix}=
\begin{cases}
\begin{pmatrix}\frac y{x+y}&-1-\frac y{x+y}\\0&y+x\end{pmatrix}
\begin{pmatrix}1&0\\1&1\end{pmatrix}
\begin{pmatrix}1&-\frac y{x+y}\\0&1\end{pmatrix}
&\text{if }v(x+y)\leq v(y),\\
\begin{pmatrix}\frac {y\varpi^r}{x+y}&-1\\0&x\end{pmatrix}
\begin{pmatrix}1&0\\\varpi^r&1\end{pmatrix}
\begin{pmatrix}\frac {x+y}{x\varpi^r}&0\\0&1\end{pmatrix}
&\text{if }v(\frac{x+y}y)=r>0.
\end{cases}\end{equation}
Notice that $v(\tfrac{x+y}y)=r$ 
if and only if $x\in y(-1+\varpi^r\cO_F^\times)$,
and hence $v(\tfrac{x+y}y)>0$ 
if and only if $x\in y(-1+\varpi\cO_F)$.
In particular $v(\tfrac{x+y}y)>0$ implies $\tfrac xy,
\tfrac {x+y}{x\varpi^r}\in\cO_F^\times$.
The calculation of 
$W_{\pi_1}\left(a(y)(\begin{smallmatrix}1&0\\1&1\end{smallmatrix})\right)$
and $W_{\pi_2}$
follows that in \cite[Lemma 5.12]{humphries2020random}.

When $0<j<m$,
let $r=v(\tfrac yx+\varpi^j)$. We have that
\begin{equation}\label{matrixdecomp.j>0}
g=\begin{cases}
\begin{pmatrix}\frac y{y+x\varpi^j}&-\frac y{y+x\varpi^j}-\varpi^j\\
0&y+x\varpi^j\end{pmatrix}
\begin{pmatrix}1&0\\1&1\end{pmatrix}
\begin{pmatrix}1&\frac x{y+x\varpi^j}-1\\0&1\end{pmatrix}
&\text{if }r\leq 0,\\
\begin{pmatrix}\frac {y\varpi^r}{y+x\varpi^j}&-1\\0&x\end{pmatrix}
\begin{pmatrix}1&0\\\varpi^r&1\end{pmatrix}
\begin{pmatrix}(\frac yx+\varpi^j)\varpi^{-r}&0\\0&1\end{pmatrix}
&\text{if }r>0.
\end{cases}\end{equation}
By the definition of $W_\pi$,
\begin{align*}
W_{\pi_1}\left(a(y)\begin{pmatrix}1&0\\\varpi^j&1\end{pmatrix}\right)
&=\frac{\zeta_F(2)^{1/2}}{\zeta_F(1)}\int_{v(x)\geq v(y)}
\omega_1(\frac {y}{y+x\varpi^j})\omega_2(y+x\varpi^j)
\left|\frac y{(y+x\varpi^j)^2}\right|^{1/2}\psi(-x)\ dx\\
&=\frac{\zeta_F(2)^{1/2}}{\zeta_F(1)}
\omega_1(y)|y|^{-1/2}\int_{v(x)\geq v(y)}
\omega_1^{-1}\omega_2(y+x\varpi^j)\psi(-x)\ dx.
\end{align*}
For $j>0$ define $U_j=1+\fp^j$. 
Let $x=(u-1)y\varpi^{-j}$. We have $y+x\varpi^j=yu$ and
\[r\leq 0\Leftrightarrow v(yu)\leq v((u-1)y\varpi^{-j})
\Leftrightarrow v(\tfrac{u-1}{u})\geq j\Leftrightarrow u^{-1}\in U_j
\Leftrightarrow u\in U_j,\]
and then
\begin{align*}
W_{\pi_1}\left(a(y)\begin{pmatrix}1&0\\\varpi^j&1\end{pmatrix}\right)
&=\frac{\zeta_F(2)^{1/2}}{\zeta_F(1)}
\omega_1(y)|y|^{-1/2}\int_{U_j}
\omega_1^{-1}\omega_2(yu)\psi(-(u-1)y\varpi^{-j})|y\varpi^{-j}|\ du\\
&=\frac{\zeta_F(2)^{1/2}}{\zeta_F(1)}
\omega_2(y)|y|^{1/2}\psi(y\varpi^{-j})q^j
\int_{U_j}\omega_1^{-1}\omega_2(u)\psi(-y\varpi^{-j}u)\ du.
\end{align*}
By the following lemma, 
$W_{\pi_1}\left(a(y)(\begin{smallmatrix}1&0\\\varpi^j&1\end{smallmatrix})\right)=0$
unless $v(y)=j-m$,
in which case
\begin{align*}
W_{\pi_1}\left(a(y)\begin{pmatrix}1&0\\\varpi^j&1\end{pmatrix}\right)
&=\frac{\zeta_F(2)^{1/2}}{\zeta_F(1)}
\omega_2(y)q^{-(j-m)/2}\psi(y\varpi^{-j})q^j
q^{-m}\sum_{b\in U_j/U_m}\omega_1^{-1}\omega_2(b)\psi(-by\varpi^{-j})\\
&=\frac{\zeta_F(2)^{1/2}}{\zeta_F(1)}
\omega_2(y)|y|^{-1/2}\psi(y\varpi^{-j})
\sum_{b\in U_j/U_m}\omega_1^{-1}\omega_2(b)\psi(-by\varpi^{-j}).
\end{align*}
\end{proof}

\begin{Lemma}
Let $\psi$ be an unramified additive character of $F$
and $\omega$ a ramified character of $F^\times$ with level $c(\omega)$.
For a positive integer $j$ let $U_j=1+\fp^j$ be a subgroup of $\cO_F^\times$.
Then for $0<j<c(\omega)$,
\[\int_{U_j}
\omega(u)\psi(au)\ du=
\begin{cases}
q^{-c(\omega)}\sum\limits_{b\in U_j/U_{c(\omega)}}\omega(b)\psi(ab)
&\text{if }v(a)= -c(\omega),\\
0&\text{if }v(a)\neq -c(\omega).
\end{cases}\]
\end{Lemma}
\begin{proof}
We follow the proof of \cite[Lemma 1.1.1]{schmidt02remarks}.
Write $u=bu'$ for $b\in U_j/U_{r}$ and $u'\in U_{r}$.
Then 
\[\int_{U_j}
\omega(u)\psi(au)\ du=
\sum_{b\in U_j/U_{r}}\omega(b)
\int_{U_{r}}\omega(u')\psi(abu')\ du'.\]

If $v(a)\leq -c(\omega)$, 
we take $r=c(\omega)$ and then $\omega(u)=1$.
The inner integral becomes
\[\int_{U_{r}}\psi(abu')\ du'
=\psi(ab)\int_{\fp^{r}}\psi(abz)\ dz.\]
It vanishes when $v(a)<-c(\omega)$. And when 
$v(a)=-c(\omega)$ it equals
$\psi(ab)\int_{\fp^{r}}\ dz=q^{-c(\omega)}\psi(ab)$.

If $v(a)> -c(\omega)$,
we take $r=c(\omega)-1$ and then $\psi(ab(u'-1))=1$ 
because that $\psi$ is unramified.
The inner integral becomes
\[\int_{U_{r}}\omega(u')\psi(ab)\psi(ab(u'-1))\ du'=
\psi(ab)\int_{U_{r}}\omega(u')\ du'=0\]
\end{proof}

%****************************************

To study the values of newforms in $\pi_3$ we have the following lemma.

\begin{Lemma}[\cite{schmidt02remarks,humphries2020random}]
\label{newform.induced}\ \\
\begin{itemize}
\item For $\pi_3=\St_{\omega_3}$ with $\omega_3$ unitary and unramified,
the newform in the induced model is
\[\varphi_{\pi_3}(g)=\begin{cases}
\omega_3(ad)\left|\dfrac ad\right|&\text{if }
g=\begin{pmatrix}a&b\\0&d\end{pmatrix}
\begin{pmatrix}1&0\\1&1\end{pmatrix}
k,\ k\in K_1(\fp),\\
-q\ \omega_3(ad)\left|\dfrac ad\right|&\text{if }
g=\begin{pmatrix}a&b\\0&d\end{pmatrix}
k,\ k\in K_1(\fp).
\end{cases}\]
Its corresponding Whittaker function has 
$W_{\pi_3}(\begin{smallmatrix}1&0\\0&1\end{smallmatrix})
=\zeta_F(2)^{-1/2}$; and for any $y\in F^\times$,
\[W_{\pi_3}(a(y))
=\zeta_F(2)^{-1/2}\cdot\begin{cases}
\omega_3(y)|y|&\text{if }v(y)\geq 0,\\
0&\text{if }v(y)< 0;\end{cases}\]
\[W_{\pi_3}\left(a(y)\begin{pmatrix}1&0\\1&1\end{pmatrix}\right)
=-\zeta_F(2)^{-1/2}\cdot\begin{cases}
q^{-1}\psi(y)\omega_3(y)|y|&\text{if }v(y)\geq -1,\\
0&\text{if }v(y)< -1.\end{cases}\]
\item For $\pi_3=\omega_3\boxplus\omega_3^{-1}$ with $\omega_3$ unitary and unramified,
the newform in the induced model is
\[\varphi_{\pi_3}(g)=
\omega_3\left(\frac ad\right)\left|\dfrac ad\right|^{1/2}
\quad\text{for }
g=\begin{pmatrix}a&b\\0&d\end{pmatrix}
k,\ k\in K;\]
\[W_{\pi_3}\begin{pmatrix}1&0\\0&1\end{pmatrix}
=\dfrac{\zeta_F(2)^{1/2}}{\zeta_F(1)L(1,\omega_3^2)}; \]
and for any $y\in F^\times$,
\[W_{\pi_3}(a(y))
=\frac{\zeta_F(2)^{1/2}}{\zeta_F(1)L(1,\omega_3^2)}\cdot\begin{cases}
|y|^{1/2}\sum\limits_{\substack{i,i'\geq 0 \\i+i'=v(y)}}
\omega_3(\varpi^i)\omega_3^{-1}(\varpi^{i'})&\text{if }v(y)\geq 0,\\
0&\text{if }v(y)< 0.\end{cases}\]
\end{itemize}
\end{Lemma}

Notice that $\left(\pi_3(\begin{smallmatrix}\varpi^{-l}&0\\0&1\end{smallmatrix})
W_{\pi_3}\right)$ is $K_1(\fp^{c(\pi_3)+l})$-invariant.
To study the oldforms we need the values of
\[\left(\pi_3\begin{pmatrix}\varpi^{-l}&0\\0&1\end{pmatrix}
W_{\pi_3}\right)
\left(a(y)\begin{pmatrix}1&0\\\varpi^j&1\end{pmatrix}\right)\quad
\text{for }1\leq l\leq c(\chi_D)-c(\pi_3),\ 0\leq j\leq c(\pi_3)+l.\]
Actually, in this paper, only the case when $j=0$ is necessary
(see Section \ref{section.localconst.calculation}).

\begin{Lemma}\label{oldform.Whittaker}\ Let $l,j\geq 0$ be two integers.
\begin{itemize}
\item For $\pi_3=\St_{\omega_3}$, if $j\leq l$,
$\left(\pi_3(\begin{smallmatrix}\varpi^{-l}&0\\0&1\end{smallmatrix})
W_{\pi_3}\right)
\left(a(y)(\begin{smallmatrix}1&0\\\varpi^j&1\end{smallmatrix})\right)$
is equal to
\[-\zeta_F(2)^{-1/2}\cdot
\begin{cases}
\psi(y\varpi^{-j})
\omega_3(y\varpi^{-l})|y\varpi^{l-2j+1}|
&\text{if }v(y)\geq 2j-l-1\\
0&\text{if }v(y)< 2j-l-1;
\end{cases}\]
if $j>l$, it is equal to
\[\zeta_F(2)^{-1/2}\cdot
\begin{cases}
\omega_3(y\varpi^{-l})|y\varpi^{-l}|
&\text{if }v(y)\geq l\\
0&\text{if }v(y)< l.
\end{cases}\]

\item For $\pi_3=\omega_3\boxplus\omega_3^{-1}$, if $j\leq l$,
$\left(\pi_3(\begin{smallmatrix}\varpi^{-l}&0\\0&1\end{smallmatrix})
W_{\pi_3}\right)
\left(a(y)(\begin{smallmatrix}1&0\\\varpi^j&1\end{smallmatrix})\right)$
is equal to
\[\frac{\zeta_F(2)^{1/2}}{\zeta_F(1)L(1,\omega_3^2)}
\cdot\begin{cases}
\psi(y\varpi^{-j})|y\varpi^{l-2j}|^{1/2}
\sum\limits_{\substack{i,i'\geq 0\\i+i'=v(y)+l-2j}} 
\omega_3(\varpi^i)\omega_3^{-1}(\varpi^{i'})
&\text{if }v(y)\geq 2j-l\\
0&\text{if }v(y)< 2j-l;
\end{cases}\]
if $j> l$, it is equal to
\[\frac{\zeta_F(2)^{1/2}}{\zeta_F(1)L(1,\omega_3^2)}
\cdot\begin{cases}
|y\varpi^{-l}|^{1/2}
\sum\limits_{\substack{i,i'\geq 0\\i+i'=v(y)-l}} 
\omega_3(\varpi^i)\omega_3^{-1}(\varpi^{i'})
&\text{if }v(y)\geq l\\
0&\text{if }v(y)<l.
\end{cases}\]
\end{itemize}
\end{Lemma}

% Clearly the results in Lemma \ref{newform.induced}
% agree with the above Lemma.
Recall that \cite[Lemma 5.17]{humphries2020random} 
calculates the spherical case 
$\pi_3=\omega_3\boxplus\omega_3^{-1}$ for $l=1$, $j=0$.

\begin{proof}
One can verify that 
\[a(y)
\begin{pmatrix}1&0\\\varpi^j&1\end{pmatrix}
\begin{pmatrix}\varpi^{-l}&0\\0&1\end{pmatrix}
=\begin{pmatrix}y\varpi^{-l}&0\\\varpi^{j-l}&1\end{pmatrix}
=a(y\varpi^{-l})
\begin{pmatrix}1&0\\\varpi^{j-l}&1\end{pmatrix}.\]
Since $W_{\pi_3}$ is $K_1(\fp)$-invariant in both cases, we have
\[\left(\pi_3\begin{pmatrix}\varpi^{-l}&0\\0&1\end{pmatrix}
W_{\pi_3}\right)
\left(a(y)\begin{pmatrix}1&0\\\varpi^j&1\end{pmatrix}\right)
=W_{\pi_3}\big(a(y\varpi^{-l})\big)\quad
\text{when $j> l$}.\]

When $j\leq l$ we have the following Iwasawa decomposition
\begin{equation}\label{eqn.Iwasawa}
\begin{pmatrix}1&0\\\varpi^{j-l}&1\end{pmatrix}
=\begin{pmatrix}\varpi^{l-j}&1-\varpi^{l-j}\\0&\varpi^{j-l}\end{pmatrix}
\begin{pmatrix}1&0\\1&1\end{pmatrix}
\begin{pmatrix}1&\varpi^{l-j}-1\\0&1\end{pmatrix},
\end{equation}
and then
\[\begin{aligned}
\begin{pmatrix}y&0\\0&1\end{pmatrix}
\begin{pmatrix}1&0\\\varpi^j&1\end{pmatrix}
\begin{pmatrix}\varpi^{-l}&0\\0&1\end{pmatrix}
&=\begin{pmatrix}y\varpi^{-j}&y\varpi^{-l}-y\varpi^{-j}\\
0&\varpi^{j-l}\end{pmatrix}
\begin{pmatrix}1&0\\1&1\end{pmatrix}
\begin{pmatrix}1&\varpi^{l-j}-1\\0&1\end{pmatrix}\\
&\in z(\varpi^{j-l})\cdot n(y\varpi^{-j}-y\varpi^{l-2j})\cdot a(y \varpi^{l-2j})\cdot
\begin{pmatrix}1&0\\1&1\end{pmatrix}K_1(\fp).
\end{aligned}\]
By the proposition of Whittaker model $W_{\pi_3}(n(x)g)=\psi(x)W_{\pi_3}(g)$ we have
\begin{multline*}
\left(\pi_3\begin{pmatrix}\varpi^{-l}&0\\0&1\end{pmatrix}
W_{\pi_3}\right)
\left(a(y)\begin{pmatrix}1&0\\\varpi^j&1\end{pmatrix}\right)\\
=\omega_{\pi_3}(\varpi^{j-l})\psi(y\varpi^{-j})\psi^{-1}(y\varpi^{l-2j})
W_{\pi_3}\left(a(y \varpi^{l-2j})\begin{pmatrix}1&0\\1&1\end{pmatrix}\right)\quad
\text{when $j\leq l$},
\end{multline*}
where $\omega_{\pi_3}$ is the central character for $\pi_3$. 

At last one can show Lemma \ref{oldform.Whittaker} from Lemma \ref{newform.induced},
noticing that the central character of $\St_{\omega_3}$ is $\omega_3^2$,
and that $\psi^{-1}(y\varpi^{l-2j})=1$ when $v(y)\geq 2j-l$.

\end{proof}

%*********************************************************************
%*********************************************************************
%*********************************************************************

\subsection{Whittaker functions for supercuspidal representations}
\label{section.supercuspidal}

For a supercuspidal representation $\pi$ of $G=\GL_2(F)$,
given the fixed additive character $\psi$, 
the Kirillov model of $\pi$ is 
a unique realization on the space of Schwartz functions 
$\varphi_\pi\in\mathcal{S}(F^\times)$ such that
\begin{equation}\label{KirillovAction.Borel}
\left(\pi\begin{pmatrix}a&b\\0&d\end{pmatrix}\varphi_\pi\right)(x)
=\omega_{\pi}(d)\psi(bd^{-1}x)\varphi_\pi(ad^{-1}x),
\end{equation}
where $\omega_{\pi}$ is the central character for $\pi$
(which is trivial in this paper).
The Whittaker function $W_{\pi}$ corresponding to $\varphi_\pi$
satisfies
\[\varphi_\pi(y)=W_{\pi}(a(y)),\quad
W_{\pi}(g)=\left(\pi(g)\varphi_\pi\right)(1),\quad
\text{and }\langle\varphi_\pi,\tilde\varphi_\pi\rangle
=\langle W_\pi,\widetilde{W}_\pi\rangle,\]
where the invariant bilinear pairing on $\pi\otimes\tilde\pi$ on the Kirillov model 
is given by \[\langle \varphi_\pi,\tilde\varphi_\pi\rangle
:=\int_{F^\times}\varphi_\pi(y)\tilde\varphi_\pi(y)\ d^\times y.\] 
In particular we have 
$W_{\pi}(a(y)g)=\left(\pi(g)\varphi_\pi\right)(y)$
for $y\in F^\times$
and $W_{\pi}(n(x)g)=\psi(x)W(g)$ for $x\in F$.

For any function $\varphi\in\mathcal{S}(F^\times)$ 
in the Kirillov model of $\pi$
which is supported only at 
$\varpi^r\cO_F^\times$,
$\varphi(\varpi^{r}x)$ can be written as a linear combination
of characters on $\cO_F^\times$
by Fourier inversion: 
\begin{equation}\label{Fourier}
\varphi(\varpi^{r}x)=
\sum_{\nu\in\widehat{\cO_F^\times}}a_\nu(\varphi)\nu(x),
\quad\text{where }
a_\nu(\varphi):=\int_{\cO_F^\times}\varphi(\varpi^{r}x)\nu^{-1}(x)\ d^\times x.
\end{equation}
We say that $\varphi$ contains level $n$ components
if $a_\nu(\varphi)\neq 0$ for some level $n$ character $\nu$
(and that it is of level $n$ if it consists of only level $n$ components).
Obviously $\varphi$ contains level $n$ components if and only if
\[\int_{\varpi^{r}\cO_F^\times}
\varphi(x)\nu(\varpi^{-r}x)\ d^\times x\neq 0\]
for some level $n$ character $\nu$. 

\begin{Lemma}\label{levelshift}
Let $\varphi(x)\in\mathcal{S}(F^\times)$ be any function 
supported only at $\varpi^r\cO_F^\times$. We have
\[\int_{\varpi^r\cO_F^\times}\varphi(x)\psi(bx)\ d^\times x\neq 0\]
only if $\varphi$ has some level $-r-v(b)$ 
(and also level 0 if $v(b)+r\geq -1$) components.
In general, if $\varphi(x)$ is of level $n$, then
$\varphi(x)\psi(bx)$ consists 
\begin{itemize}
	\item of only level $n$ components if $v(b)> -r-\max\{n,1\}$,
	\item of only level $-r-v(b)$ components if $v(b)<-r-\max\{n,1\}$, and 
	\item of all level $\leq \max\{n,1\}$ components if $v(b)=-r-\max\{n,1\}$.
\end{itemize}
\end{Lemma}

\begin{proof}
It is sufficient to show the lemma for 
$\varphi(x)=\chi(\varpi^{-r}x)\1_{\varpi^{r}\cO_F^\times}(x)$
where $\chi$ is any level $n$ character
and $\1_{\varpi^{r}\cO_F^\times}$ is
the characteristic function of $\varpi^{r}\cO_F^\times$.
\cite[Lemma 1.1.1]{schmidt02remarks} shows that,
\begin{equation}\label{addcharint}
\int_{\varpi^m\cO_F^\times}\psi(x)\ dx
=\begin{cases}
q^{-m}\zeta_F(1)^{-1}&\text{if }m\geq 0,\\
-1&\text{if }m=-1,\\
0&\text{if }m\leq -2;
\end{cases}\end{equation}
and for any ramified character $\omega$ of $F^\times$,
\begin{equation}\label{addcharintram}
\int_{\varpi^r\cO_F^\times}
\omega^{-1}(x)\psi(x)|x|^{-s}\ dx
=\begin{cases}
\epsilon(s,\omega,\psi)&\text{if }r=-c(\omega),\\
0&\text{otherwise.}
\end{cases}\end{equation}
Then, for any character $\nu$ of $\cO_F^\times$
(we extend $\chi,\nu$ to be characters on $F^\times$ by
defining $\chi(\varpi)=\nu(\varpi)=1$),
\begin{equation}\label{FourierCoefficients}
\begin{aligned}
&\ \int_{\varpi^r\cO_F^\times}\varphi(x)\psi(bx)\nu(\varpi^{-r}x)\ d^\times x
=\int_{\varpi^r\cO_F^\times}(\chi\nu)(\varpi^{-r}x)\psi(bx)\ d^\times x\\
=&\ \zeta_F(1)\chi\nu(b\varpi^{r})^{-1}|b\varpi^r|^{-1}
\int_{b\varpi^r\cO_F^\times}\chi\nu(x)\psi(x)\ dx\\
=&\ \begin{cases}
1&\text{if }c(\chi\nu)=0,\ v(b)+r\geq 0,\\
-\frac 1{q-1}&\text{if }c(\chi\nu)=0,\ v(b)+r=-1,\\
\zeta_F(1)\chi\nu(b\varpi^{r})^{-1}\epsilon(1,\chi^{-1}\nu^{-1},\psi)
&\text{if }c(\chi\nu)\geq 1,\ v(b)+r=-c(\chi\nu),\\
0&\text{otherwise.}
\end{cases}
\end{aligned}\end{equation}
This completes the proof,
noticing that $c(\chi\nu)\leq\max\{c(\chi),c(\nu)\}$ 
and that inequality holds only if $c(\chi)= c(\nu)$.

% If $v(b)\geq -r$ then $c(\chi\nu)=0$ and then $c(\chi)= c(\nu)$.\\
% If $v(b)= -r-1$ then $c(\chi\nu)=0$ or $1$ and then 
% $c(\chi), c(\nu)\in\{0,1\}$ or $c(\chi)= c(\nu)$.\\
% If $v(b)\leq -r-2$ then $c(\chi\nu)=-r-v(b)\geq 2$. \\
% \textbullet\  If $c(\chi)=n>-r-v(b)$ (now $-n< v(b)+r <-1$) then $c(\chi)= c(\nu)$. \\
% \textbullet\  If $n<-r-v(b)$ (i.e. $v(b)+r<-\max\{n,1\}$) then $c(\nu)=-r-v(b)$. \\
% \textbullet\  If $n=-r-v(b)$ (and $n>1$) then $c(\nu)\leq n$.
\end{proof}

The Bruhat decomposition says that $G=B\cup BwN$, where 
$w=(\begin{smallmatrix}0&-1\\1&0\end{smallmatrix})$
and $B=ZAN$ is the upper triangular Borel subgroup of $G=\GL_2(F)$.
Then the action $\pi(g),g\in G$ in the Kirillov model 
can be expressed purely in terms of $\pi(w)$ and $\pi|_B$. 
We recall a fact that shows the operator $\pi(w)$
on some multiplicative characters.

\begin{Fact}[{\cite{jacquet2006automorphic}, 
\cite[Theorem 37.3]{bushnell2006local} and \cite[Proposition A.1]{hu2016cuspidal}}]
Let $\pi$ a supercuspidal representation with trivial central character.
Assume that it has conductor $\fp^{c(\pi)}$.
Let $\nu$ be a multiplicative character of $\cO_F^\times$ 
with level $c(\nu)$.
The action of $w=(\begin{smallmatrix}0&-1\\1&0\end{smallmatrix})$
in the Kirillov model of $\pi$ satisfies
\begin{equation}\label{KirillovAction}
\pi(w)(\nu(\varpi^{-r}\cdot) \1_{\varpi^r\cO_F^\times})
=C_{\nu}
\nu^{-1}(\varpi^{-r'}\cdot)\1_{\varpi^{r'}\cO_F^\times}
\end{equation}
where $C_{\nu}=\epsilon(\tfrac 12,\pi\otimes\nu^{-1},\psi)$ is independent of $r$
and $r'=-r-\max\{c(\pi),2c(\nu)\}$
(except when the residue field of $F$ is of characteristic $2$ 
and $c(\pi)=2c(\nu)\geq 4$).
In particular
\begin{equation}\label{KirillovAction1}
\pi(w)\1_{\varpi^r\cO_F^\times}
=C_{\1}\1_{\varpi^{-r-c(\pi)}\cO_F^\times},\quad
\text{where }C_{\1}=\epsilon(\tfrac 12,\pi,\psi)=\pm 1.
\end{equation}
\end{Fact}

Next we recall a lemma about the new vector in a supercuspidal representation $\pi_3$
and the values of its corresponding Whittaker function.

\begin{Lemma}[{\cite[Lemma 5.10]{hu2016cuspidal},
\cite[Corollary 2.18]{hu2017triple}}]
\label{newform}
For a supercuspidal representation $\pi_3$ with trivial central character, 
the new vector in the Kirillov model 
is $\varphi_{\pi_3}=\1_{\cO_F^\times}$,
the characteristic function of $\cO_F^\times$.
Its corresponding Whittaker function $W_{\pi_3}$ satisfies:
\begin{itemize}
\item $W_{\pi_3}(a(y))=\1_{\cO_F^\times}(y)$
for any $y\in F^\times$; 
and therefore $\langle\varphi_{\pi_3},\tilde\varphi_{\pi_3}\rangle
=\langle W_{\pi_3},\widetilde{W}_{\pi_3}\rangle=1$.
\item For $0\leq j<c(\pi_3)$, 
$W_{\pi_3}\left(a(y)(\begin{smallmatrix}1&0\\\varpi^j&1\end{smallmatrix})\right)$
is supported only at $v(y) =\min\{0,2j-c(\pi_3)\}$,
and it consists of only level $c(\pi_3)-j$ 
(and also level $0$ if $j=c(\pi_3)-1$) components;
the exception happens when the residue field of $F$ is of characteristic $2$ 
(the central character is assumed to be trivial in this paper),
$c(\pi_3)\geq 4$ is an even number and $j = c(\pi_3)/2$,
in which case
$W_{\pi_3}\left(a(y)(
\begin{smallmatrix}1&0\\\varpi^{c(\pi_3)/2}&1\end{smallmatrix})\right)$ 
is supported at $v(y)\geq 0$, consisting of level $c(\pi_3)/2$ components.
\item Moreover we have
\[\int_{v(y) =\min\{0,2j-c(\pi_3)\}}
W_{\pi_3}\left(a(y)\begin{pmatrix}1&0\\\varpi^j&1\end{pmatrix}\right)
\ d^\times y=\begin{cases}
1&\text{if }j\geq c(\pi_3),\\
-\frac 1{q-1}&\text{if }j= c(\pi_3)-1,\\
0&\text{otherwise};
\end{cases}\]
\[\int_{v(y) =\min\{0,2j-c(\pi_3)\}}
W_{\pi_3}\left(a(y)\begin{pmatrix}1&0\\\varpi^j&1\end{pmatrix}\right)
\psi(-\varpi^{-j}y)\ d^\times y=\begin{cases}
C_{\1}&\text{if }j=0,\\
-\frac 1{q-1}C_{\1}&\text{if }j= 1,\\
0&\text{otherwise},
\end{cases}\]
where $C_{\1}=\epsilon(\tfrac 12,\pi_3,\psi)=\pm 1$.
\end{itemize}
\end{Lemma}

Next we generalize the above lemma.

\begin{Lemma}\label{sc.new.int}
With assumptions and notations in the above lemma,
\begin{itemize}
\item for $j\geq c(\pi_3)$,
\[\int_{\cO_F^\times}
W_{\pi_3}\left(a(y)\begin{pmatrix}1&0\\\varpi^j&1\end{pmatrix}\right)
\psi(by)\ d^\times y
=\begin{cases}
1&\text{if }b\in\cO_F,\\
-\frac 1{q-1}&\text{if }b\in\varpi^{-1}\cO_F^\times,\\
0&\text{otherwise;}
\end{cases}\]
\item for $0\leq j<c(\pi_3)$, in general, 
\[\int_{v(y) = \min\{0, 2j-c(\pi_3)\}}
W_{\pi_3}\left(a(y)\begin{pmatrix}1&0\\\varpi^j&1\end{pmatrix}\right)
\psi(by)\ d^\times y\]
vanishes unless $v(b)=-\min\{j,c(\pi_3)-j\}$
(or $b\in\varpi^{-1}\cO_F$ when $j=c(\pi_3)-1$);
in the exceptional case when the residue field of $F$ is of characteristic $2$,
$c(\pi_3)\geq 4$ is even and $j = c(\pi_3)/2$, the integral
$\int_{v(y) = r}
W_{\pi_3}\left(a(y)(\begin{smallmatrix}1&0\\\varpi^j&1\end{smallmatrix})\right)
\psi(by)\ d^\times y$
vanishes unless $r\geq 0$ and $v(b)=-r-c(\pi_3)/2$;
\item in particular for $j=0$,
\[\int_{\varpi^{-c(\pi_3)}\cO_F^\times}
W_{\pi_3}\left(a(y)\begin{pmatrix}1&0\\1&1\end{pmatrix}\right)\psi(by)
\ d^\times y
=\begin{cases}
C_{\1}&\text{if }b\in -1+\varpi^{c(\pi_3)}\cO_F,\\
-\frac {C_{\1}}{q-1}&\text{if }b\in -1+\varpi^{c(\pi_3)-1}\cO_F^\times,\\
0&\text{otherwise,}
\end{cases}\]
where $C_{\1}=\epsilon(\tfrac 12,\pi_3,\psi)=\pm 1$.
\end{itemize}
\end{Lemma}

\begin{proof}
When $b=0$ or $b=-\varpi^{-j}$ the lemma is 
\cite[Lemma 5.10(2)(3)]{hu2016cuspidal}.

Recall that 
\[W_{\pi}\left(a(y)\begin{pmatrix}1&0\\\varpi^j&1\end{pmatrix}\right)
=\left(\pi\begin{pmatrix}1&0\\\varpi^j&1\end{pmatrix}\varphi_\pi\right)(y).\]
When $j\geq c(\pi_3)$, 
$W_{\pi_3}\left(a(y)(\begin{smallmatrix}1&0\\
\varpi^j&1\end{smallmatrix})\right)=\1_{\cO_F^\times}(y)$ 
is simply the new vector. 
The integral is equal to
\begin{equation}\label{Lemma1.1.1reformulation}
\int_{\cO_F^\times}\psi(by)\ d^\times y
=\zeta_F(1)|b|^{-1}\int_{b\cO_F^\times}\psi(y)\ dy
=\begin{cases}
1&\text{if }v(b)\geq 0,\\
-\frac 1{q-1}&\text{if }v(b)=-1,\\
0&\text{if }v(b)\leq -2
\end{cases}\end{equation}
(see (\ref{addcharint}) for the last step).

When $0\leq j<c(\pi_3)$,
Lemma \ref{newform} shows that, in the general case, 
$W_{\pi_3}\left(a(y)(\begin{smallmatrix}1&0\\
\varpi^j&1\end{smallmatrix})\right)$ is supported 
only at $v(y) =r'=\min\{0,2j-c(\pi_3)\}$,
consists of only level $c(\pi_3)-j$ 
(and also level 0 if $j=c(\pi_3)-1$) components.
Then we apply Lemma \ref{levelshift} and see that,
the integral we study vanishes
unless $v(b)=-\min\{j,c(\pi_3)-j\}$
(or $v(b)\geq -1$ when $j\geq c(\pi_3)-1$).
One can study the exceptional case using the same argument.

Recall that for any $\varphi\in\mathcal{S}(F^\times)$ 
in the Kirillov model of $\pi$ we have by \eqref{KirillovAction.Borel} that
\[\left(\pi\begin{pmatrix}1&b\\0&1\end{pmatrix}\varphi\right)(x)
=\psi(bx)\varphi(x).\]
Then we can write 
\[\begin{split}
W_{\pi_3}\left(a(y)\begin{pmatrix}1&0\\\varpi^j&1\end{pmatrix}\right)\psi(by)
&=\left(\pi\begin{pmatrix}1&b\\0&1\end{pmatrix}
\left(\pi\begin{pmatrix}1&0\\\varpi^j&1\end{pmatrix}\1_{\cO_F^\times}\right)
\right)(y)\\
&=\left(\pi\begin{pmatrix}1+b\varpi^j&b\\\varpi^j&1\end{pmatrix}\1_{\cO_F^\times}
\right)(y).
\end{split}\]
In particular when $j=0$ (and $v(b)=0$) we can decompose the matrix as
\[\begin{pmatrix}1+b&b\\1&1\end{pmatrix}
=\begin{pmatrix}1&1+b\\0&1\end{pmatrix}w
\begin{pmatrix}1&1\\0&1\end{pmatrix}.\]
Recall that \eqref{KirillovAction} gives 
the action of $w=(\begin{smallmatrix}0&-1\\1&0\end{smallmatrix})$
in the Kirillov model. Then one can show that
\[\begin{split}
&\ W_{\pi_3}\left(a(y)\begin{pmatrix}1&0\\1&1\end{pmatrix}\right)\psi(by)
=\left(\pi_3(\begin{pmatrix}1&1+b\\0&1\end{pmatrix}w
\begin{pmatrix}1&1\\0&1\end{pmatrix})\1_{\cO_F^\times}
\right)(y)\\
=&\ \left(\pi\begin{pmatrix}1&1+b\\0&1\end{pmatrix}
(\pi_3(w)\1_{\cO_F^\times})\right)(y)
=\left(\pi\begin{pmatrix}1&1+b\\0&1\end{pmatrix}
(C_{\1}\1_{\varpi^{-c(\pi_3)}\cO_F^\times})\right)(y)\\
=&\ C_{\1}\psi((1+b)y)\1_{\varpi^{-c(\pi_3)}\cO_F^\times}(y).
\end{split}\]
By (\ref{Lemma1.1.1reformulation}) we have 
\begin{multline*}
\int_{v(y)=-c(\pi_3)}
W_{\pi_3}\left(a(y)\begin{pmatrix}1&0\\1&1\end{pmatrix}\right)\psi(by)
\ d^\times y\\
=C_{\1}\int_{\varpi^{-c(\pi_3)}\cO_F^\times}\psi((1+b)y)\ d^\times y
=\begin{cases}
C_{\1}&\text{if }v(1+b)\geq c(\pi_3),\\
-\frac {C_{\1}}{q-1}&\text{if }v(1+b)=c(\pi_3)-1,\\
0&\text{if }v(1+b)\leq c(\pi_3)-2.
\end{cases}
\end{multline*}
\end{proof}

To study the oldforms we need the values of
$\left(\pi_3(\begin{smallmatrix}\varpi^{-l}&0\\0&1\end{smallmatrix})
W_{\pi_3}\right)
\left(a(y)(\begin{smallmatrix}1&0\\\varpi^j&1\end{smallmatrix})\right)$
for $1\leq l\leq c(\chi_D)-c(\pi_3)$.

\begin{Lemma}\label{oldform.Whittaker.supercuspidal}
With assumptions and notations in the above lemma, for an integer $l\geq 0$,
\begin{itemize}
\item if $j\geq c(\pi_3)+l$, then 
\[\left(\pi_3\begin{pmatrix}\varpi^{-l}&0\\0&1\end{pmatrix}
W_{\pi_3}\right)
\left(a(y)\begin{pmatrix}1&0\\\varpi^{j}&1\end{pmatrix}\right)
=\1_{\varpi^{l}\cO_F^\times}(y)\]
and 
\[\int_{v(y)=l}
\left(\pi_3\begin{pmatrix}\varpi^{-l}&0\\0&1\end{pmatrix}
W_{\pi_3}\right)
\left(a(y)\begin{pmatrix}1&0\\\varpi^{j}&1\end{pmatrix}\right)
\psi(by)\ d^\times y
=\begin{cases}
1&\text{if }b\in \varpi^{-l}\cO_F,\\
-\frac 1{q-1}&\text{if }b\in \varpi^{-l-1}\cO_F^\times,\\
0&\text{otherwise.}
\end{cases}\]
\item if $0\leq j\leq l$,
$\left(\pi_3(\begin{smallmatrix}\varpi^{-l}&0\\0&1\end{smallmatrix})
W_{\pi_3}\right)
\left(a(y)(\begin{smallmatrix}1&0\\\varpi^j&1\end{smallmatrix})\right)$
is supported only at $v(y)=2j-l-c(\pi_3)$,
and
\[\int_{v(y)=2j-l-c(\pi_3)}
\left(\pi_3\begin{pmatrix}\varpi^{-l}&0\\0&1\end{pmatrix}
W_{\pi_3}\right)
\left(a(y)\begin{pmatrix}1&0\\\varpi^{j}&1\end{pmatrix}\right)
\psi(by)\ d^\times y\]
is equal to 
\[\begin{cases}
C_{\1}&\text{if }b\in -\varpi^{-j}+\varpi^{c(\pi_3)+l-2j}\cO_F,\\
-\frac {C_{\1}}{q-1}&
\text{if }b\in -\varpi^{-j}+\varpi^{c(\pi_3)+l-2j-1}\cO_F^\times,\\
0&\text{otherwise}
\end{cases}\]
where $C_{\1}=\epsilon(\tfrac 12,\pi_3,\psi)=\pm 1$.
\end{itemize}
\end{Lemma}

\begin{proof}
When $j> l$ we have 
\[\begin{pmatrix}1&0\\\varpi^j&1\end{pmatrix}
\begin{pmatrix}\varpi^{-l}&0\\0&1\end{pmatrix}
=\begin{pmatrix}\varpi^{-l}&0\\\varpi^{j-l}&1\end{pmatrix}
=\begin{pmatrix}\varpi^{-l}&0\\0&1\end{pmatrix}
\begin{pmatrix}1&0\\\varpi^{j-l}&1\end{pmatrix}.\]
If $j-l\geq c(\pi_3)$ we have
\[\left(\pi_3\begin{pmatrix}\varpi^{-l}&0\\0&1\end{pmatrix}
W_{\pi_3}\right)
\left(a(y)\begin{pmatrix}1&0\\\varpi^{j}&1\end{pmatrix}\right)
=\pi_3\begin{pmatrix}\varpi^{-l}&0\\0&1\end{pmatrix}
\1_{\cO_F^\times}(y)
=\1_{\varpi^{l}\cO_F^\times}(y)\]
and 
\begin{multline*}
\int_{v(y)=l}
\left(\pi_3\begin{pmatrix}\varpi^{-l}&0\\0&1\end{pmatrix}
W_{\pi_3}\right)
\left(a(y)\begin{pmatrix}1&0\\\varpi^{j}&1\end{pmatrix}\right)
\psi(by)\ d^\times y\\
=\int_{v(y)=l}\psi(by)\ d^\times y
=\begin{cases}
1&\text{if }v(b)\geq -l,\\
-\frac 1{q-1}&\text{if }v(b)=-l-1,\\
0&\text{if }v(b)\leq -l-2.
\end{cases}
\end{multline*}

When $j\leq l$,
by \eqref{eqn.Iwasawa} we have the Iwasawa decomposition
\[\begin{pmatrix}1&0\\\varpi^j&1\end{pmatrix}
\begin{pmatrix}\varpi^{-l}&0\\0&1\end{pmatrix}
=\begin{pmatrix}\varpi^{-j}&\varpi^{-l}-\varpi^{-j}\\
0&\varpi^{j-l}\end{pmatrix}
\begin{pmatrix}1&0\\1&1\end{pmatrix}
\begin{pmatrix}1&\varpi^{l-j}-1\\0&1\end{pmatrix}.\]
And by \eqref{KirillovAction.Borel} we have 
\[\begin{split}
&\ \left(\pi_3\begin{pmatrix}\varpi^{-l}&0\\0&1\end{pmatrix}
W_{\pi_3}\right)
\left(a(y)\begin{pmatrix}1&0\\\varpi^{j}&1\end{pmatrix}\right)\\
=&\ \pi_3(\begin{pmatrix}\varpi^{-j}&\varpi^{-l}-\varpi^{-j}\\
0&\varpi^{j-l}\end{pmatrix}
\begin{pmatrix}1&0\\1&1\end{pmatrix}
\begin{pmatrix}1&\varpi^{l-j}-1\\0&1\end{pmatrix})\1_{\cO_F^\times}(y)\\
=&\ \pi_3\begin{pmatrix}\varpi^{-j}&\varpi^{-l}-\varpi^{-j}\\
0&\varpi^{j-l}\end{pmatrix}
(\pi_3\begin{pmatrix}1&0\\1&1\end{pmatrix}\1_{\cO_F^\times})(y)\\
=&\ \psi((\varpi^{-l}-\varpi^{-j})\varpi^{l-j}y)
(\pi_3\begin{pmatrix}1&0\\1&1\end{pmatrix}\1_{\cO_F^\times})(\varpi^{l-2j}y)\\
=&\ W_{\pi_3}\left(a(\varpi^{l-2j}y)
\begin{pmatrix}1&0\\1&1\end{pmatrix}\right)
\psi((\varpi^{-j}-\varpi^{l-2j})y).
\end{split}\]
By Lemma \ref{newform}
it is supported only at $v(y)=2j-l-c(\pi_3)$,
and 
\[\begin{split}
&\ \int_{v(y)=2j-l-c(\pi_3)}
\left(\pi_3\begin{pmatrix}\varpi^{-l}&0\\0&1\end{pmatrix}
W_{\pi_3}\right)
\left(a(y)\begin{pmatrix}1&0\\\varpi^{j}&1\end{pmatrix}\right)
\psi(by)\ d^\times y\\
=&\ \int_{v(y)=2j-l-c(\pi_3)}
W_{\pi_3}\left(a(\varpi^{l-2j}y)
\begin{pmatrix}1&0\\1&1\end{pmatrix}\right)
\psi((\varpi^{-j}-\varpi^{l-2j}+b)y)\ d^\times y\\
=&\ \int_{v(y')=-c(\pi_3)}
W_{\pi_3}\left(a(y')
\begin{pmatrix}1&0\\1&1\end{pmatrix}\right)
\psi((\varpi^{j-l}-1+b\varpi^{2j-l})y')\ d^\times y'\\
=&\ \begin{cases}
C_{\1}&\text{if }b\in -\varpi^{-j}+\varpi^{c(\pi_3)+l-2j}\cO_F,\\
-\frac {C_{\1}}{q-1}&\text{if }b\in -\varpi^{-j}+\varpi^{c(\pi_3)+l-2j-1}\cO_F^\times,\\
0&\text{otherwise}
\end{cases}
\end{split}\]
by Lemma \ref{sc.new.int}.
\end{proof}

% ********************************************************************

\subsection{The adjoint lift of a supercuspidal representation}
\label{L(Ad)}

A supercuspidal representation $\pi$ of $G=\GL_2(F)$
is called unramified 
(i.e. of Type-1 as we have defined in Section \ref{section.globalcalculation}) 
if $\pi\simeq\pi\otimes\eta$
for some unramified character $\eta\neq\1$ of $F^\times$.
(By comparing central characters one can see that $\eta$ is quadratic.)
Let $\eta$ be the (nontrivial) unramified quadratic character of $F^\times$.
\cite[Corollary 1.3]{gelbart1978relation} gives the $L$-factor of 
the adjoint lift of a supercuspidal representation:
\[L(s,\pi,\Ad)=\begin{cases}
1&\text{if }\pi\not\simeq\pi\otimes\eta,\\
(1+q^{-s})^{-1}&\text{if }\pi\simeq\pi\otimes\eta.\end{cases}\]

In the case when the residue field of $F$ has characteristic $p=2$,
the ``unramification'' of a supercuspidal representation $\pi$ 
is actually equivalent to its ``dihedralness''.
Recall that $\pi$ is called dihedral 
(cf. \cite[Theorem 4.8.6]{bump1998automorphic}) 
if it is associated with 
a quadratic field extension $E/F$ and a character of $E^\times$
that is not trivial on the kernel of the norm map $N_{E/F}$ 
from $E^\times$ to $F^\times$.
(One can also find the construction in \cite[Section 19]{bushnell2006local}.)
The Tame Parametrization Theorem (cf. \cite[Section 20.1]{bushnell2006local})
says that every supercuspidal representation $\pi$ of $\GL_2(F)$ is dihedral 
if the residue characteristic of $F$ is an odd prime;
but when the residue characteristic is $p=2$, only the unramified ones
have such correspondence:
$\pi$ is supercuspidal and unramified 
if and only if it is ``unramified'' dihedral,
i.e. it is associated with 
an unramified quadratic field extension $E/F$.
(The equivalence of ``unramified supercuspidal'' 
and ``unramified dihedral'' is also true when $p\neq 2$, 
cf. \cite[Section 20.3]{bushnell2006local}.)
This explains the assumption of the Maass form $f$ in Theorem \ref{WatsonIchinoCM}.

% ********************************************************************
% ********************************************************************
% ********************************************************************
\subsection{Local constants in the Watson--Ichino formula}
\label{section.localconst.calculation}

To apply Lemma \ref{Lemma.LocalConstCalculation} to the calculation 
of the local Rankin--Selberg integral $\lRS$, 
we need the following result.

\begin{Lemma}\label{Lemma.BKinvariance}
Fix an unramified additive character $\psi$ of $F$. 
Let $\pi_1=\omega_1\boxplus\omega_2$, 
$\pi_2=\omega_2^{-1}\boxplus\omega_1^{-1}$
be principal series representations of $G=\GL_2(F)$
with $\omega_1,\omega_2$ both unitary,
$c(\omega_1)=m>0$ and $c(\omega_2)=0$.
Let $\pi_3$ be a generic representation of $G$ with trivial central character
and $c(\pi_3)\leq m$.
Then, for $\varphi_1$ in the induced model of $\pi_1$, 
$W_2\in\mathcal{W}(\pi_2,\bar\psi)$ and
$W_3\in\mathcal{W}(\pi_3,\psi)$, 
the function $\Theta:K\to \CC$ defined by
\[\Theta(k):=\int_{F^\times}
\varphi_1(a(y)k)W_2(a(y)k)W_3(a(y)k)
\ \frac{d^\times y}{|y|}\]
is left $(B\cap K)$-invariant and right $K_1(\fp^m)$-invariant.
\end{Lemma}

\begin{proof}
Any $b\in B\cap K$ can be decomposed as $b=z(t)a(y')n(x)$ with $t,y'\in\cO_F^\times$ and $x\in\cO_F$,
and we have that
\[a(y)bk=\begin{pmatrix}y&\\&1\end{pmatrix}\begin{pmatrix}t&\\&t\end{pmatrix}
\begin{pmatrix}y'&\\&1\end{pmatrix}\begin{pmatrix}1&x\\&1\end{pmatrix}k\\
=\begin{pmatrix}t&\\&t\end{pmatrix}\begin{pmatrix}yy'&yy'x\\&1\end{pmatrix}k=z(t)n(yy'x)a(yy')k.\]
% \[\begin{aligned}
% a(y)bk&=\begin{pmatrix}y&\\&1\end{pmatrix}\begin{pmatrix}t&\\&t\end{pmatrix}
% \begin{pmatrix}y'&\\&1\end{pmatrix}\begin{pmatrix}1&x\\&1\end{pmatrix}k\\
% &=\begin{pmatrix}t&\\&t\end{pmatrix}\begin{pmatrix}yy'&yy'x\\&1\end{pmatrix}k=z(t)n(yy'x)a(yy')k.
% \end{aligned}\]
Recall that, the action of $z(t)=(\begin{smallmatrix}t&0\\0&t\end{smallmatrix})$ is given by the central character:
\[\varphi_1(z(t)g)=\omega_1\omega_2(t)\varphi_1(g),\quad W_2(z(t)g)=\omega_2^{-1}\omega_1^{-1}(t)W_2(g),\quad W_3(z(t)g)=W_3(g);\]
and the action of $n(x)=(\begin{smallmatrix}1&x\\0&1\end{smallmatrix})$ 
is given by proposition of induced model or Whittaker model respectively:
\[\varphi_1(n(x)g)=\varphi_1(g),\quad W_2(n(x)g)=\bar\psi(x)W_2(g),\quad W_3(n(x)g)=\psi(x)W_3(g).\]
Therefore
\[\begin{aligned}
\varphi_1\big(a(y)bk\big)&=\omega_1\omega_2(t)&&\varphi_1(a(yy')k),\\
W_2\big(a(y)bk\big)&=\omega_2^{-1}\omega_1^{-1}(t)&\!\!\!\!\!\bar\psi(yy'x)&W_2(a(yy')k),\\
W_3\big(a(y)bk\big)&=&\!\!\!\!\!\psi(yy'x)&W_3(a(yy')k);
\end{aligned}\]
and hence 
\[\Theta(bk)=\int_{F^\times}
\varphi_1(a(yy')k)W_2(a(yy')k)W_3(a(yy')k)
\ \frac{d^\times y}{|y|}
=|y'|\Theta(k)=\Theta(k)\]
for any $b\in B\cap K$, with $|y'|=1$ since $y'\in\cO_F^\times$.

At last the assumptions on the conductors of these three representations
imply the right $K_1(\fp^m)$-invariance of $\varphi_1,W_2,W_3$, and thus of $\Theta(k)$.
\end{proof}

The above lemma still holds if $\pi_i,$ $i=1,2,3$, are all generic with level $c(\pi_i)\leq m$,
and with central character $\omega_{\pi_i}$
such that $\omega_{\pi_1}\omega_{\pi_2}\omega_{\pi_3}=\1$.

By the definition of $\lRS$ together with Lemma \ref{Lemma.LocalConstCalculation} and Lemma \ref{Lemma.BKinvariance}, 
$\lRS(\varphi_{\pi_1}, W_{\pi_2}, W_{\pi_3})$ is equal to
\[\zeta_F(1)^{1/2}\sum_{j=0}^m A_j
\int_{F^\times}
\varphi_{\pi_1}\left(a(y)\begin{pmatrix}1&0\\\varpi^j&1\end{pmatrix}\right)
W_{\pi_2}\left(a(y)\begin{pmatrix}1&0\\\varpi^j&1\end{pmatrix}\right)
W_{\pi_3}\left(a(y)\begin{pmatrix}1&0\\\varpi^j&1\end{pmatrix}\right)
\ \frac{d^\times y}{|y|}.\]
Recall that, by Lemma \ref{Lemma.newform1}, the new vector in the induced model of 
$\pi_1=\omega_1\boxplus\omega_2$,
where $c(\omega_1)=m>c(\omega_2)=0$, 
satisfies
\[\varphi_{\pi_1}\left(a(y)\begin{pmatrix}1&0\\\varpi^j&1\end{pmatrix}\right)
=\begin{cases}
\omega_1(y)|y|^{1/2}&\text{if }j=0,\\
0&\text{if }0<j\leq m.
\end{cases}\]
This means we only need to work on the case with $j=0$. The integral becomes
\[\lRS(\varphi_{\pi_1}, W_{\pi_2}, W_{\pi_3})
=\zeta_F(1)^{1/2} \frac{\zeta_F(2)}{\zeta_F(1)}
\int_{F^\times}
\omega_1(y)|y|^{1/2}
W_{\pi_2}\left(a(y)\begin{pmatrix}1&0\\1&1\end{pmatrix}\right)
W_{\pi_3}\left(a(y)\begin{pmatrix}1&0\\1&1\end{pmatrix}\right)
\ \frac{d^\times y}{|y|}.\]
By Lemma \ref{Lemma.newform1.Whittaker} we have
\begin{align*}&\ 
\lRS(\varphi_{\pi_1}, W_{\pi_2}, W_{\pi_3})
% =\zeta_F(1)^{1/2}\int_K\int_{F^\times}
% \varphi_{\pi_1}(a(y)k)W_{\pi_2}(a(y)k)W_{\pi_3}(a(y)k)
% \ \frac{d^\times y}{|y|}\ dk
\\
% =&\ \zeta_F(1)^{1/2}\sum_{j=0}^m A_j
% \int_{F^\times}
% \varphi_{\pi_1}\left(a(y)\begin{pmatrix}1&0\\\varpi^j&1\end{pmatrix}\right)
% W_{\pi_2}\left(a(y)\begin{pmatrix}1&0\\\varpi^j&1\end{pmatrix}\right)
% W_{\pi_3}\left(a(y)\begin{pmatrix}1&0\\\varpi^j&1\end{pmatrix}\right)
% \ \frac{d^\times y}{|y|}\\
% =&\ \zeta_F(1)^{1/2} \frac{\zeta_F(2)}{\zeta_F(1)}
% \int_{F^\times}
% \varphi_{\pi_1}\left(a(y)\begin{pmatrix}1&0\\1&1\end{pmatrix}\right)
% W_{\pi_2}\left(a(y)\begin{pmatrix}1&0\\1&1\end{pmatrix}\right)
% W_{\pi_3}\left(a(y)\begin{pmatrix}1&0\\1&1\end{pmatrix}\right)
% \ \frac{d^\times y}{|y|}\\
% =&\ \frac{\zeta_F(2)}{\zeta_F(1)^{1/2}}\sum_{r}\int_{v(y)=r}
% \omega_1(y)|y|^{1/2}
% W_{\pi_2}\left(a(y)\begin{pmatrix}1&0\\1&1\end{pmatrix}\right)
% W_{\pi_3}\left(a(y)\begin{pmatrix}1&0\\1&1\end{pmatrix}\right)
% \ \frac{d^\times y}{|y|}\\
=&\ \frac{\zeta_F(2)}{\zeta_F(1)^{1/2}}\sum_{r\geq -m}\int_{v(y)=r}
\omega_1(y)|y|^{1/2}\\
&\qquad\qquad\qquad\qquad\qquad
\cdot\frac{\zeta_F(2)^{1/2}}{\zeta_F(1)}
\omega_1^{-1}(y)|y|^{1/2}\psi(-y)
\epsilon(1,\omega_1^{-1}\omega_2,\psi)
W_{\pi_3}\left(a(y)\begin{pmatrix}1&0\\1&1\end{pmatrix}\right)
\ \frac{d^\times y}{|y|}\\
=&\ \frac{\zeta_F(2)^{3/2}}{\zeta_F(1)^{3/2}}
\epsilon(1,\omega_1^{-1}\omega_2,\psi)
\sum_{r\geq -m}\int_{v(y)=r}\psi(-y)
W_{\pi_3}\left(a(y)\begin{pmatrix}1&0\\1&1\end{pmatrix}\right)
\ d^\times y.
\end{align*}
To study oldforms one only need to replace $W_{\pi_3}$ with 
$\pi_3(\begin{smallmatrix}\varpi^{-l}&0\\0&1\end{smallmatrix})
W_{\pi_3}$ for $l\geq 1$.

% ********************************************************************

\subsubsection{Proof of Proposition \ref{localconst.special}}

For $0\leq l\leq m-1$ 
(so the result works for both newforms and oldforms) and 
$\pi_3=\St_{\omega_3}$ with $\omega_3^2=\1$, 
Lemmas \ref{newform.induced} and \ref{oldform.Whittaker} show that
\[
\left(\pi_3\begin{pmatrix}\varpi^{-l}&0\\0&1\end{pmatrix}
W_{\pi_3}\right)
\left(a(y)\begin{pmatrix}1&0\\1&1\end{pmatrix}\right)
=-\zeta_F(2)^{-1/2}\cdot
\begin{cases}
q^{-1}\psi(y)
\omega_3(y\varpi^{-l})|y\varpi^{l}|
&\text{if }v(y)\geq -l-1\\
0&\text{if }v(y)\leq -l-2.
\end{cases}\]
Therefore
\begin{align*}
&\ \lRS(\varphi_{\pi_1}, W_{\pi_2}, 
\pi_3\begin{pmatrix}\varpi^{-l}&0\\0&1\end{pmatrix}
W_{\pi_3})\\
=&\ \frac{\zeta_F(2)^{3/2}}{\zeta_F(1)^{3/2}}
\epsilon(1,\omega_1^{-1}\omega_2,\psi)
\sum_{r\geq -l-1}\int_{v(y)=r}\psi(-y)
\zeta_F(2)^{-1/2}(-q^{-1})\psi(y)\omega_3(y\varpi^{-l})|y\varpi^{l}|
\ d^\times y\\
=&\ -q^{-1}\frac{\zeta_F(2)}{\zeta_F(1)^{3/2}}
\epsilon(1,\omega_1^{-1}\omega_2,\psi)\omega_3(\varpi^{-2l})
\sum_{r'\geq -1}\int_{v(y')=r'}\omega_3(y')|y'|\ d^\times y'\quad(y'=y\varpi^{l})\\
=&\ -q^{-1}\frac{\zeta_F(2)}{\zeta_F(1)^{3/2}}
\epsilon(1,\omega_1^{-1}\omega_2,\psi)
\omega_3(\varpi^{-2l})
\sum_{r\geq -1}\int_{\cO_F^\times}\omega_3(\varpi^ru)q^{-r}\ d^\times u\\
=&\ -q^{-1}\frac{\zeta_F(2)}{\zeta_F(1)^{3/2}}
\epsilon(1,\omega_1^{-1}\omega_2,\psi)
\omega_3(\varpi^{-2l})
\sum_{r\geq -1}\omega_3(\varpi)^rq^{-r}\\
% =&\ -q^{-1}\frac{\zeta_F(2)}{\zeta_F(1)^{3/2}}
% \epsilon(1,\omega_1^{-1}\omega_2,\psi)
% \omega_3(\varpi^{-2l})
% \frac{\omega_3^{-1}(\varpi)q}{1-\omega_3(\varpi)q^{-1}}\\
=&\ -\frac{\zeta_F(2)}{\zeta_F(1)^{3/2}}
\epsilon(1,\omega_1^{-1}\omega_2,\psi)
L(1,\omega_3)
\omega_3(\varpi^{-2l-1}),\quad\text{ for }0\leq l\leq m-1.
\end{align*}
By \eqref{addcharintram} one has
\[|\epsilon(1,\omega_1^{-1}\omega_2,\psi)|
=|\epsilon(\tfrac 12,\omega_1^{-1}\omega_2,\psi)
q^{-c(\omega_1^{-1}\omega_2)/2}|
=q^{-m/2}.\]
So the numerator in the local constant is
\[I(\varphi\otimes\tilde\varphi)=q^{-m}\frac{\zeta_F(2)^2}{\zeta_F(1)^3}
L(1,\omega_3)^2.\]

The denominator in $I'(\varphi\otimes\tilde\varphi)$ is given by
\[\langle\varphi,\tilde\varphi\rangle
=\langle W_{\pi_1},\widetilde{W}_{\pi_1}\rangle
\langle W_{\pi_2},\widetilde{W}_{\pi_2}\rangle
\langle W_{\pi_3},\widetilde{W}_{\pi_3}\rangle.\]
By definition we have
\[\langle W_{\pi_1},\widetilde{W}_{\pi_1}\rangle
=\int_{v(y)\geq 0}\left|\frac{\zeta_F(2)^{1/2}}{\zeta_F(1)}
\omega_2(y)|y|^{1/2}\right|^2\ d^\times y
=\frac{\zeta_F(2)}{\zeta_F(1)^2}\int_{\cO_F}|y|\ d^\times y
=\frac{\zeta_F(2)}{\zeta_F(1)},\]
and hence $\langle W_{\pi_1},\widetilde{W}_{\pi_1}\rangle
=\langle W_{\pi_2},\widetilde{W}_{\pi_2}\rangle
=\zeta_F(2)/\zeta_F(1)$; also
\[\langle W_{\pi_3},\widetilde{W}_{\pi_3}\rangle
=\int_{v(y)\geq 0}\left|\zeta_F(2)^{-1/2}
\omega_3(y)|y|\right|^2\ d^\times y
=\frac{1}{\zeta_F(2)}
\int_{\cO_F}|y|^2\ d^\times y
=1.\]
We get that
\[\frac{I(\varphi\otimes\tilde\varphi)}
{\langle\varphi,\tilde\varphi\rangle}
=\frac{q^{-m}\frac{\zeta_F(2)^2}{\zeta_F(1)^3}L(1,\omega_3)^2}
{\left(\frac{\zeta_F(2)}{\zeta_F(1)}\right)^2\cdot 1}
=q^{-m}\frac{L(1,\omega_3)^2}{\zeta_F(1)}.\]

The local $L$-factors are defined 
in the same way as in \cite{humphries2020random}:
in this case
\begin{gather*}
L(s,\pi_1\otimes\pi_2\otimes\pi_3)
=L(s+\tfrac 12,\omega_3)^2,\\
L(s,\pi_1,\Ad)=L(s,\pi_2,\Ad)=\zeta_F(s),\quad
L(s,\pi_3,\Ad)=\zeta_F(s+1).
\end{gather*}
We can simplify that 
\[I'(\varphi\otimes\tilde\varphi)
=q^{-m}\frac{\zeta_F(1)}{\zeta_F(2)}
=q^{-m}(1+q^{-1}).\]

% ********************************************************************

\subsubsection{Proof of Proposition \ref{localconst.unram}}

For $\pi_3=\omega_3\boxplus\omega_3^{-1}$ and $0\leq l\leq m$, we have shown that
\begin{multline*}
\left(\pi_3\begin{pmatrix}\varpi^{-l}&0\\0&1\end{pmatrix}
W_{\pi_3}\right)
\left(a(y)\begin{pmatrix}1&0\\1&1\end{pmatrix}\right)\\
=\frac{\zeta_F(2)^{1/2}}{\zeta_F(1)L(1,\omega_3^2)}
\begin{cases}
\psi(y)|\varpi^l y|^{1/2}\sum\limits_{i+i'=v(y)+l}
\omega_3(\varpi^i)\omega_3^{-1}(\varpi^{i'})
&\text{if }v(y)\geq -l,\\
0&\text{if }v(y)< -l.
\end{cases}
\end{multline*}
Then
\begin{align*}&\ 
\lRS(\varphi_{\pi_1}, W_{\pi_2}, 
\pi_3\begin{pmatrix}\varpi^{-l}&0\\0&1\end{pmatrix}W_{\pi_3})\\
=&\ \frac{\zeta_F(2)^{3/2}}{\zeta_F(1)^{3/2}}
\epsilon(1,\omega_1^{-1}\omega_2,\psi)
\sum_{r\geq -m}\int_{v(y)=r}\psi(-y)
\left(\pi_3
\begin{pmatrix}\varpi^{-l}&0\\0&1\end{pmatrix}
W_{\pi_3}\right)
\left(a(y)\begin{pmatrix}1&0\\1&1\end{pmatrix}\right)
\ d^\times y\\
=&\ \frac{\zeta_F(2)^2}{\zeta_F(1)^{5/2}}
\epsilon(1,\omega_1^{-1}\omega_2,\psi)
\sum_{r\geq -l}\int_{v(y)=r}
L(1,\omega_3^2)^{-1}
|\varpi^l y|^{1/2}\sum_{i+i'=v(y)+l}
\omega_3(\varpi^i)\omega_3^{-1}(\varpi^{i'})
\ d^\times y\\
=&\ \frac{\zeta_F(2)^2}{\zeta_F(1)^{5/2}}
\epsilon(1,\omega_1^{-1}\omega_2,\psi)
L(1,\omega_3^2)^{-1}
\sum_{r'\geq 0}\int_{v(y')=r'}
|y'|^{1/2}\sum_{i+i'=r'}
\omega_3(\varpi^i)\omega_3^{-1}(\varpi^{i'})
\ d^\times y'\quad(y'=y\varpi^l)\\
=&\ \frac{\zeta_F(2)^2}{\zeta_F(1)^{5/2}}
\epsilon(1,\omega_1^{-1}\omega_2,\psi)
L(1,\omega_3^2)^{-1}
\sum_{r'\geq 0}q^{-r'/2}\sum_{i+i'=r'}
\omega_3(\varpi^i)\omega_3^{-1}(\varpi^{i'})\\
=&\ \frac{\zeta_F(2)^2}{\zeta_F(1)^{5/2}}
\epsilon(1,\omega_1^{-1}\omega_2,\psi)
L(1,\omega_3^2)^{-1}
\sum_{i\geq 0}\sum_{i'\geq 0}
\omega_3(\varpi^i)q^{-i/2}\omega_3^{-1}(\varpi^{i'})q^{-i'/2}\\
=&\ \frac{\zeta_F(2)^2}{\zeta_F(1)^{5/2}}
\epsilon(1,\omega_1^{-1}\omega_2,\psi)
L(1,\omega_3^2)^{-1}
(1-\omega_3(\varpi)q^{1/2})^{-1}(1-\omega_3^{-1}(\varpi)q^{1/2})^{-1}\\
=&\ \frac{\zeta_F(2)^2}{\zeta_F(1)^{5/2}}
\epsilon(1,\omega_1^{-1}\omega_2,\psi)
L(1,\omega_3^2)^{-1}
L(\tfrac 12,\omega_3)L(\tfrac 12,\omega_3^{-1}).
\end{align*}
We now get the numerator
\[I(\varphi\otimes\tilde\varphi)=q^{-m}\frac{\zeta_F(2)^4}{\zeta_F(1)^5}
L(1,\omega_3^2)^{-1}L(1,\omega_3^{-2})^{-1}
L(\tfrac 12,\omega_3)^2L(\tfrac 12,\omega_3^{-1})^2.\]

For the denominator, 
the normalization of newform $\varphi_{\pi_3}$ implies 
$\langle\varphi_{\pi_3},\tilde\varphi_{\pi_3}\rangle=1$.
Then \[\langle W_{\pi_3},\widetilde{W}_{\pi_3}\rangle=1.\]
Recall that
\[\langle W_{\pi_1},\widetilde{W}_{\pi_1}\rangle=
\langle W_{\pi_2},\widetilde{W}_{\pi_2}\rangle
=\frac{\zeta_F(2)}{\zeta_F(1)};\]
therefore
\[\frac{I(\varphi\otimes\tilde\varphi)}
{\langle\varphi,\tilde\varphi\rangle}
% =\frac{q^{-m}\frac{\zeta_F(2)^4}{\zeta_F(1)^5}
% L(1,\omega_3^2)^{-1}L(1,\omega_3^{-2})^{-1}
% L(\tfrac 12,\omega_3)^2L(\tfrac 12,\omega_3^{-1})^2}
% {\left(\frac{\zeta_F(2)}{\zeta_F(1)}\right)^2
% \cdot 1}
=q^{-m}\frac{\zeta_F(2)^2
L(\tfrac 12,\omega_3)^2L(\tfrac 12,\omega_3^{-1})^2}
{\zeta_F(1)^3L(1,\omega_3^2)L(1,\omega_3^{-2})}.\]

Recall that the local $L$-factors are given by 
\[L(s,\pi_1\otimes\pi_2\otimes\pi_3)
=L(s,\omega_3)^2L(s,\omega_3^{-1})^2,\]
\[L(s,\pi_1,\Ad)=L(s,\pi_2,\Ad)=\zeta_F(s),\]
\[L(s,\pi_3,\Ad)=\zeta_F(s)L(s,\omega_3^2)L(s,\omega_3^{-2}).\]
One can simplify that
\[I'(\varphi\otimes\tilde\varphi)
=q^{-m}.\]

% ********************************************************************

\subsubsection{Proof of Proposition \ref{localconst.supercuspidal}}

For $\pi_3$ supercuspidal
(with $2\leq c(\pi_3)\leq c(\chi_D)=m$), 
$W_{\pi_3}\left(a(y)(\begin{smallmatrix}1&0\\1&1\end{smallmatrix})\right)$
is supported only at $v(y) =-c(\pi_3)$.
By Lemma \ref{sc.new.int} we see
\[\int_{v(y)=-c(\pi_3)}
W_{\pi_3}\left(a(y)\begin{pmatrix}1&0\\1&1\end{pmatrix}\right)\psi(-y)
\ d^\times y
=C_{\1},\]
where $C_{\1}=\epsilon(\tfrac 12,\pi_3,\psi)=\pm 1$.
Moreover, for the oldforms, we have shown 
in Lemma \ref{oldform.Whittaker.supercuspidal} that
$\left(\pi_3(\begin{smallmatrix}\varpi^{-l}&0\\0&1\end{smallmatrix})
W_{\pi_3}\right)
\left(a(y)(\begin{smallmatrix}1&0\\1&1\end{smallmatrix})\right)$
is supported only at $v(y)=-l-c(\pi_3)$,
and for $l\geq 0$ we also have
\[
\int_{v(y)=-l-c(\pi_3)}
\left(\pi_3\begin{pmatrix}\varpi^{-l}&0\\0&1\end{pmatrix}
W_{\pi_3}\right)
\left(a(y)\begin{pmatrix}1&0\\1&1\end{pmatrix}\right)
\psi(-y)\ d^\times y=C_{\1}.
\]
So in both cases we have
\[\lRS(\varphi_{\pi_1}, W_{\pi_2}, 
\pi_3\begin{pmatrix}\varpi^{-l}&0\\0&1\end{pmatrix}
W_{\pi_3})
=\frac{\zeta_F(2)^{3/2}}{\zeta_F(1)^{3/2}}\cdot
\epsilon(1,\omega_1^{-1}\omega_2,\psi^{-1})\cdot (\pm 1).\]
The numerator is
\[I(\varphi\otimes\tilde\varphi)=
q^{-m}\frac{\zeta_F(2)^3}{\zeta_F(1)^3}.\]

For the denominator, recall that
\[\langle W_{\pi_1},\widetilde{W}_{\pi_1}\rangle=
\langle W_{\pi_2},\widetilde{W}_{\pi_2}\rangle=\frac{\zeta_F(2)}{\zeta_F(1)},
\quad \langle W_{\pi_3},\widetilde{W}_{\pi_3}\rangle=1;\]
and hence 
\begin{equation}\label{eqn.proveagain}
\frac{I(\varphi\otimes\tilde\varphi)}
{\langle\varphi,\tilde\varphi\rangle}
=q^{-m}\frac{\zeta_F(2)}{\zeta_F(1)}.\end{equation}

The local L-factors are given by
\[L(s,\pi_1\otimes\pi_2\otimes\pi_3)
=1;\quad
L(s,\pi_1,\Ad)=L(s,\pi_2,\Ad)=\zeta_F(s);\]
let $\eta$ be the (nontrivial) unramified quadratic character of $F^\times$.
By \cite[Corollary 1.3]{gelbart1978relation} we know
\[L(s,\pi_3,\Ad)=\begin{cases}
1&\text{if }\pi_3\not\simeq\pi_3\otimes\eta,\\
(1+q^{-s})^{-1}&\text{if }\pi_3\simeq\pi_3\otimes\eta.
\end{cases}\]
One can simplify that
\[I'(\varphi\otimes\tilde\varphi)
=q^{-m}(1+q^{-1})L(1,\pi_3,\Ad)
=\begin{cases}
q^{-m}(1+q^{-1})&\text{if }\pi_3\not\simeq\pi_3\otimes\eta,\\
q^{-m}&\text{if }\pi_3\simeq\pi_3\otimes\eta.
\end{cases}\]

% ********************************************************************
% ********************************************************************
% ********************************************************************

\subsubsection{Direct calculation by matrix coefficients}
One can also calculate the local constants 
by the methods used in \cite{hu2017triple}, 
which is to calculate directly the matrix coefficients: 
by definition,
\[\frac{I(\varphi\otimes\tilde\varphi)}
{\langle\varphi,\tilde\varphi\rangle}=
\int_{Z(F)\backslash\GL_2(F)}
\Phi_{\pi_1}(g)\Phi_{\pi_2}(g)\Phi_{\pi_3}(g)\ dg,\]
where $\Phi_{\pi}$ denotes the normalized matrix coefficient
\[\Phi_{\pi}(g)=\frac 1{\langle W_\pi,\widetilde{W}_\pi\rangle}
\int_{F^\times} 
W_\pi\left(a(y)g\right)
\widetilde{W}_\pi(a(y))
\ d^\times y.\]

% *********************************************************
We consider the case when $\pi_3$ is supercuspidal for example.
Let $b = (\begin{smallmatrix}y&x\\0&1\end{smallmatrix})=n(x)a(y)$ 
for some $y\in F^\times$, $x \in F$.
For $\pi_1,\pi_2$, 
one can generalize the results in \cite{humphries2020random} 
and show that 
\begin{itemize}
	\item $\Phi_{\pi_1}(b)$ is equal to
\[\begin{cases}
\omega_2(y)|y|^{1/2}&v(y)\geq 0,\ v(x)\geq 0\\
\omega_2(y)|y|^{-1/2}&v(y)\leq 0,\ v(x)\geq v(y)\\
0&\text{otherwise}.
\end{cases}\]
	\item $\Phi_{\pi_1}(b(\begin{smallmatrix}1&0\\1&1\end{smallmatrix}))$ is equal to
\[\begin{cases}
\omega_1(y)|y|^{1/2}
\omega_1^{-1}\omega_2(x+y)|x+y|^{-1}&
\text{if }v(x+y)\leq \min\{-m,v(y)\}\\
0&\text{otherwise}.
\end{cases}\]
	\item For $0<j<m$, 
	$\Phi_{\pi_1}(b(\begin{smallmatrix}1&0\\\varpi^j&1\end{smallmatrix}))$ is equal to
\[\begin{cases}
\omega_2(y)|y|^{-1/2}\omega_1^{-1}\omega_2(1+x\varpi^j/y)&
\text{if }v(y)\leq j-m,\ v(x)\geq v(y)+m-j-1,\\
0&\text{otherwise.}
\end{cases}\]
\end{itemize}
And $\Phi_{\pi_2}$ is the complex conjugation of $\Phi_{\pi_1}$.
% *********************************************************
% *********************************************************
For $\pi_3$ supercuspidal, \cite[Proposition 2.19]{hu2017triple} shows:
\begin{itemize}
	\item $\Phi_{\pi_3}(b)$ is supported on $v(y)=0$ and $v(x)\geq -l-1$;
		\[\Phi_{\pi_3}(b)=\begin{cases}
		1&\text{if }v(y)=0,\ v(x)\geq -l,\\
		-\frac 1{q-1}&\text{if }v(y)=0,\ v(x)= -l-1;\end{cases}\]
	\item for $j=c(\pi_3)+l-1$, 
		$\Phi_{\pi_3}\left(b(\begin{smallmatrix}1&0\\\varpi^j&1\end{smallmatrix})\right)$
		is supported on $v(y)=0$, 
		$v(x)\geq -l-1$, and 
		\[\Phi_{\pi_3}
		\left(b\begin{pmatrix}1&0\\\varpi^{c(\pi_3)+l-1}&1\end{pmatrix}\right)
		=-\frac 1{q-1}\quad\text{if }v(x)\geq -l;\]
		\[\int_{v(x)=-l-1}\Phi_{\pi_3}
		\left(b\begin{pmatrix}1&0\\\varpi^{c(\pi_3)+l-1}&1\end{pmatrix}\right)\ dx
		=\frac {q^l}{q-1}\quad\text{if }v(x)\geq -l;\]
	\item for $c(\pi_3)<4$ and $0\leq j<c(\pi_3)+l-1$,
		$\Phi_{\pi_3}
		\left(b(\begin{smallmatrix}1&0\\\varpi^j&1\end{smallmatrix})\right)$
		is supported on $v(y)=\min\{0,2j-c(\pi_3)-2l\}$, 
		$v(x)=j-c-2l$;
		it is of level $c(\pi_3)+l-j$ as a function in $y$.
\end{itemize}

Since $W_\pi$ is right $K_1(\fp^{m})$-invariant, 
\cite[Lemma 2.2]{hu2016cuspidal}
implies that $I(\varphi\otimes\tilde\varphi)
/\langle\varphi,\tilde\varphi\rangle$ is equal to
\[\sum_{j=0}^m A_j\int_{Z(F)\backslash B(F)}
\prod_{i=1}^3\Phi_{\pi_i}\left(b
\begin{pmatrix}1&0\\\varpi^j&1\end{pmatrix}
\right)\ db,
\quad %\text{where }
A_j=\frac{\zeta_F(2)}{\zeta_F(1)}\cdot
\begin{cases}
1&\text{if }j=0\\q^{-j}\zeta_F(1)^{-1}&\text{if }0<j<m\\
q^{-m}&\text{if }j=m,
\end{cases}\]
where $b = n(x)a(y)$ 
and $db = |y|^{-1}d^\times y\ dx$.
Here
\[\begin{split}
\prod_{i=1}^3\Phi_{\pi_i}\left(b\right)
&=\begin{cases}
|y|\Phi_{\pi_3}&v(y)\geq 0,\ v(x)\geq 0\\
|y|^{-1}\Phi_{\pi_3}&v(y)\leq 0,\ v(x)\geq v(y)\\
0&\text{otherwise}
\end{cases}\\
&=\begin{cases}
1&v(y)= 0,\ v(x)\geq 0\\
0&\text{otherwise}.
\end{cases}
\end{split}\]
So
\[\int_{Z(F)\backslash B(F)}
\prod_{i=1}^3\Phi_{\pi_i}\left(b\right)\ db
=\int_{v(y)= 0}\frac{d^\times y}{|y|}\int_{v(x)\geq 0}dx
=1.\]

Next we show that all other terms (when $0\leq j<m$) vanish. For $j=0$,
\[\prod_{i=1}^3\Phi_{\pi_i}\left(b
\begin{pmatrix}1&0\\1&1\end{pmatrix}\right)
=\begin{cases}
\dfrac{|y|}{|x+y|^{2}}\Phi_{\pi_3}&
\text{if }v(x+y)\leq \min\{-m,v(y)\},\\
0&\text{otherwise},
\end{cases}\]
with
\[\Phi_{\pi_3}\left(b
\begin{pmatrix}1&0\\1&1\end{pmatrix}\right)
=\begin{cases}
C_{\1}&\text{if }v(y)=- c-2l,\ 
x\in -y+\varpi^{-l}\cO_F,\\
-\frac {C_{\1}}{q-1}&\text{if }v(y)=- c-2l,\ 
x\in -y+\varpi^{-l-1}\cO_F^\times,\\
0&\text{otherwise}
\end{cases}\]
is supported only on $v(y)=-c-2l$, $v(x+y)\geq -l-1$
(so $v(x+y)\leq v(y)$ cannot happen).
Therefore \[\prod_{i=1}^3\Phi_{\pi_i}\left(b
\begin{pmatrix}1&0\\1&1\end{pmatrix}\right)=0.\]

For $c+l-1\leq j<m$, $\Phi_{\pi_3}\left(b
(\begin{smallmatrix}1&0\\\varpi^j&1\end{smallmatrix})\right)$ 
is supported on $v(y)=0$. 
Since $j-m<0$, \[\prod_{i=1}^3\Phi_{\pi_i}\left(b
\begin{pmatrix}1&0\\\varpi^j&1\end{pmatrix}\right)=0.\]

For $0<j<c+l-1$,
\[\left.\begin{array}{r}
v(y)\leq j-m\ (\text{which is }<0)\\v(y)=\min\{0,2j-c-2l\}
\end{array}\right\}\Rightarrow
2j-c-2l\leq j-m
\Rightarrow j\leq c+2l-m,
\]
which is possible only if $l>\tfrac 12 (m-c)$ 
(recall that $0\leq l\leq m-c$).
Under this assumption $\Phi_{\pi_3}\left(b
(\begin{smallmatrix}1&0\\\varpi^j&1\end{smallmatrix})\right)$ 
is supported on
\[v(y)=2j-c-2l,\quad v(x) = j - c - 2l;\]
but in this case $v(x)\geq v(y)+m-j-1$ does not hold.
That means, on the support of $\Phi_{\pi_3}\left(b
(\begin{smallmatrix}1&0\\\varpi^j&1\end{smallmatrix})\right)$,
$\Phi_{\pi_1}\left(b
(\begin{smallmatrix}1&0\\\varpi^j&1\end{smallmatrix})\right)$ vanishes.
So again
\[\prod_{i=1}^3\Phi_{\pi_i}\left(b
\begin{pmatrix}1&0\\\varpi^j&1\end{pmatrix}\right)=0.\]

In conclusion
\[\frac{I(\varphi\otimes\tilde\varphi)}
{\langle\varphi,\tilde\varphi\rangle}
=\frac{\zeta_F(2)}{\zeta_F(1)}
q^{-m}\int_{Z(F)\backslash B(F)}
\prod_{i=1}^3\Phi_{\pi_i}\left(b\right)\ db
=q^{-m}\frac{\zeta_F(2)}{\zeta_F(1)}.\]
One can continue from \eqref{eqn.proveagain} and complete the proof.

% *********************************************************
\section*{Acknowledgements}
The author would like to thank Bingrong Huang for suggesting this problem, 
and to thank Zihao Wang and Hongbo Yin for helpful discussions. 
The author also thanks the anonymous referees 
for making helpful comments on an earlier version
which led to improvement of the exposition.
This research was completed while the author was supported by 
the National Key Research and Development Program of China (No. 2021YFA1000700).

\bibliographystyle{alpha}
\bibliography{CMFormsBib}
\end{document}